\documentclass[10pt,a4paper]{article}
\usepackage[latin1]{inputenc}
\usepackage{amsmath}
\usepackage{amsfonts}
\usepackage{amssymb}
\usepackage{amsthm}
\usepackage{mathrsfs}
\usepackage{latexsym}
\usepackage{stmaryrd}
\usepackage[all]{xy}

\newtheorem{definition}{Definition}

\newtheorem{corollary}[definition]{Corollary}
\newtheorem{theorem}[definition]{Theorem}
\newtheorem{lemma}[definition]{Lemma}
\newtheorem{proposition}[definition]{Proposition}

\newtheorem*{main}{Theorem}

\def\XXint#1#2#3{{\setbox0=\hbox{$#1{#2#3}{\int}$ }
\vcenter{\hbox{$#2#3$ }}\kern-.6\wd0}}

\begin{document}
\section*{Automorphism groups of simplicial complexes and rigidity for uniformly bounded representations} 
\begin{Large}
{Juhani Koivisto} \footnote{Supported by V\"ais\"al\"a Foundation and the Academy of Finland, project 252293.}
\end{Large} \\ \emph{Department of Mathematics and Statistics, University of Helsinki, Finland, email: juhani.koivisto@helsinki.fi}


\paragraph{Abstract.}
We consider $L^p$-cohomology of reflexive Banach spaces and give a spectral condition implying the vanishing of 1-cohomology with coefficients in uniformly bounded representations on a Hilbert space. 
\paragraph*{Mathematics Subject Classification (2000):} 20F65
\paragraph*{Key words} Fixed point property, cohomology, Banach space, uniformly bounded representation, spectral criterion.

\section{Introduction}
Since its introduction by David Kazhdan in \cite{Kazhdan}, property $(T)$ and its generalizations as cohomological vanishing has become a fundamental concept in mathematics \cite{BHV}. The aim of this paper is to extend the framework of W. Ballman and J. \'Swi\c atkowski \cite{BS} to reflexive Banach spaces and as an application, to give a spectral condition implying vanishing of cohomology for uniformly bounded representations on a Hilbert space.
Along with W. Ballman and J. \'Swi\c atkowski, A. \.{Z}uk \cite{Zuk96} was among the first to also consider such criteria for unitary representations, both following fundamental work by H. Garland [G].  Since then, extending the spectral method beyond Hilbert spaces has been considered in \cite{Monod, Cha, Dymara, Ers, Fish, L} and by Piotr W. Nowak \cite{Nowak} extending the spectral method in \cite{Zuk03} to reflexive Banach spaces. Appropriately extending the scheme in \cite{Nowak} we similarly extend the spectral condition of \cite{BS} to uniformly bounded representations on a Hilbert space. Motivation for such generalizations arises, among others, from Shalom's conjecture \cite{OWR} stating that any hyperbolic group $\Gamma$ admits a uniformly bounded representation $\pi$ with ${H^1}(\Gamma, \pi) \neq 0$ together with a proper cocycle in $Z^1(\Gamma, \pi)$. \\
For a finite graph $K$ with vertices $\mathcal{V}_K$, consider the graph Laplacian $\triangle_+$ on the space of real valued functions on $\mathcal{V}_K$ defined by $$\triangle_+f(v) = f(v)-Mf(v),$$ where $Mf(v)$ is the mean value of $f$ on the vertices adjacent to $v$. Denote by $\lambda_1(K)$ the spectral gap of $\triangle_+$ and its associated Poincar\'e constant by $\kappa_2(K, \mathbb{R})=\lambda_1(K)^{-1/2}$. More generally $\kappa_2(K, \mathcal{H})= \lambda_1(K)^{-1/2}$ for any separable infinite-dimensional Hilbert space $\mathcal{H}$, \cite{Nowak}.


\begin{main}
Let $X$ be a locally finite $2$-dimensional simplicial complex, $\Gamma$ a discrete properly discontinuous group of automorphisms of $X$ and $\pi : \Gamma \rightarrow \mathrm{B}(\mathcal{H})$ a uniformly bounded representation of $\Gamma$ on a separable infinite-dimensional Hilbert space $\mathcal{H}$. If for any vertex $\tau$ of $X$ the link $X_\tau$ is connected and $$\sup_{g \in \Gamma} \Vert \pi_g\Vert < \dfrac{\sqrt{2}}{ \kappa_2(X_\tau, \mathcal{H})},$$ then $L^2H^1(X,\pi) = 0$.
\end{main}

\subparagraph*{Structure of the paper} In Sections \ref{SU} to \ref{PI} the framework of \cite{BS} for unitary representations on Hilbert spaces is extended to reflexive Bananch spaces and isometric representations: Section \ref{SU} introduces the generalized set up; Section \ref{projsec} and \ref{L} deal with the dual of the twisted cochains;  Section \ref{DC} introduces differentials and codifferentials; Section \ref{secloc} discusses localization of the problem and Section \ref{PI} introduces the spectral set up in terms of Poincar\'e inequalities and constants on the links. Section \ref{LH} introduces $L^p$-cohomology as a natural extension to $L^2$-cohomology, which is then applied to uniformly bounded representations using the fact that they correspond to isometric representations on some reflexive Banach space.

\subparagraph*{Acknowledgements} I would like to thank Piotr W. Nowak for suggesting this topic, invaluable advice, and devotion without which this project would not have been possible. I would also like to thank V\"ais\"al\"a foundation, my advisor Ilkka Holopainen and the ''Analysis, metric geometry and differential and metirc topology'' project for financial support, Pekka Pankka and Izhar Oppenheim for helpful discussions and correspondence, and Antti Per\"al\"a for many enjoyable conversations on related topics. 

\section{Set up} \label{SU}
In this chapter notation is fixed. We recall the notation and some basic facts used by \cite{BS} for weighted simplicial complexes and extend the notion of square integrable cochains to reflexive Banach spaces and $p > 1$. 
\subsection{Weighted complexes}
Throughout, let $X$ denote an $n$-dimensional locally finite simplicial complex. Following \cite{BS} we use the following notation: $X(k)$ is the set of (unordered) $k$-simplexes of $X$; $\Sigma(k)$ is the set of ordered $k$-simplexes of $X$. As usual we write $\sigma = \lbrace v_0, \dots, v_k \rbrace$ for a $k$-simplex and $\sigma =(v_0, \dots, v_k)$ for an ordered $k$-simplex. If the vertices of $\tau \in \Sigma(l)$ are vertices of $\sigma \in \Sigma(k)$, we say that $\tau \subset \sigma$, and for $\tau = (v_0, \dots, \hat{v}_i, \dots, v_k)$, i.e. $v_i \notin \tau$, we denote by $[\sigma : \tau] = (-1)^i$ the sign of $\tau$ in $\sigma = (v_0, \dots, v_k)$. As customary, we write $\sigma_i$ for $(v_0, \dots, \hat{v}_i, \dots, v_k)$. In addition to orientation we consider $X$ to be equipped with a weight $\omega$, by which we mean a map from the oriented simplexes of $X$ to the integers such that for $\sigma = (v_0, \dots, v_k) \in \Sigma(k)$, $$\omega(\sigma) = \omega(\lbrace v_0, \dots, v_k \rbrace),$$ where $\omega(\lbrace v_0, \dots, v_k \rbrace)$ denotes the number of $n$-simplexes containing $\lbrace v_0, \dots, v_k\rbrace$. In addition, we assume that $\omega(\sigma) \geq 1$ for every simplex of $X$. Beginning from Section \ref{secloc} and onwards, we consider $X$ locally through its links, where, by the link of $\tau = (v_0, \dots, v_l) \in \Sigma(l)$ denoted by $X_\tau$, we mean the $(n-l-1)$-dimensional subcomplex consisting of all simplexes $\lbrace w_0, \dots, w_j\rbrace$ disjoint from $\tau$ such that $\lbrace v_0, \dots, v_l\rbrace \cup \lbrace w_0, \dots, w_j\rbrace$ is a simplex of $X$. Since $X$ is locally finite, $X_\tau$ is finite. Here as previously, $X_\tau(j)$ denotes the $j$-simplexes of $X_\tau$, $\Sigma_\tau(j)$ its oriented $j$-simplexes and so on. In particular, for $\sigma \in \Sigma_\tau(j)$ and $\tau \in \Sigma(l)$ we denote by $\sigma * \tau \in \Sigma(j+l+1)$ the join of $\sigma$ and $\tau$ obtained by juxtaposing the two in that order. \\
In addition to the above, we assume throughout that $X$ is a $\Gamma$-space where $\Gamma$ is a discrete topological group acting properly and discontinuously by simplicial automorphisms on $X$. In other words, $\Gamma$ permutes the simplexes of $X$ preserving their order and weights: that is for $\sigma = (v_0, \dots, v_k) \in \Sigma(k)$, $g \cdot \sigma = (g(v_0), \dots, g(v_k)) \in \Sigma(k)$ and $\omega (\sigma) = \omega(g \cdot \sigma)$. As usual, we denote by $\Gamma \sigma$ and $\Gamma_\sigma$ the $\Gamma$-orbit and stabilizer of $\sigma \in \Sigma(k)$, respectively, by $\Sigma(k, \Gamma) \subset \Sigma(k)$ some chosen set of representatives of $\Gamma$-orbits in $\Sigma(k)$, and by $\vert \cdot \vert$ the counting measure on $\Gamma$. In particular since $\Gamma$ is discrete, stabilizers are finite and the Haar measure on $\Gamma$ is $\vert \cdot \vert$. Although the discreteness assumption can be avoided, it will be used when constructing projections in Section \ref{projsec}. For the following frequently used facts we refer to \cite{BS}:

\begin{proposition} \label{combinatorial} \cite{BS}
Let $n$ be the dimension of $X$. Then, for  $\tau \in \Sigma(k)$ $$\sum_{\substack{\sigma \in \Sigma(k+1) \\ \tau \subset \sigma}}\omega(\sigma) = (n-k)(k+2)!\omega(\tau).$$ 
\end{proposition} \qed

\begin{proposition} \label{switchingsums} \cite{BS}
For $0 \leq l < k \leq n$, let $f=f(\tau, \sigma)$ be a $\Gamma$-invariant function on the set of pairs $(\tau, \sigma)$, $\tau \in \Sigma(l)$, $\sigma \in \Sigma(k)$, such that $\tau \subset \sigma$. Then $$\sum_{\sigma \in \Sigma(k, \Gamma)} \sum_{\substack{\tau \in \Sigma(l) \\ \tau \subset \sigma}} \dfrac{f(\tau, \sigma)}{\vert \Gamma_\sigma \vert} =  \sum_{\tau \in \Sigma(l, \Gamma)}\sum_{\substack{\sigma \in \Sigma(k) \\ \tau \subset \sigma}}\dfrac{f(\tau, \sigma)}{\vert \Gamma_\tau \vert},$$ whenever either side is absolutely convergent. \qed
\end{proposition} 
More generally, Proposition \ref{switchingsums} holds for locally compact unimodular groups \cite{DJ00} replacing the counting measure with the Haar measure.

\subsection{Banach space setting, isometric representations and $p$-integrable cochains}

Throughout, let $(E, \Vert \cdot \Vert_E)$ denote a reflexive Banach space, $\langle \cdot, \cdot \rangle_E$ the natural pairing between $E$ and its continuous dual $E^*$, $\simeq$ isomorphism, $\cong$ isometric isomorphism and $p^*$ the adjoint index of $p$ such that $1/p+1/p^* = 1$.  
Moreover, let $\pi: \Gamma \rightarrow \mathrm{Iso}(E)$ denote an isometric representation of $\Gamma$ on $E$ where $\mathrm{Iso}(E)$ denotes the group of isometric linear automorphisms on $E$ and by $\bar{\pi}: \Gamma \rightarrow \mathrm{Iso}(E^*)$ its corresponding contragradient representation given by $\bar{\pi}_g = \pi^*_{g^{-1}}$ where $\pi^*$ is the transpose of $\pi$. For combinatorial purposes we also introduce antisymmetrization:
\begin{definition}
For $n \geq 1$ we denote by $S_n$ the symmetric group of $n$ elements and by $\mathrm{sign} \colon S_n \rightarrow \lbrace -1, 1 \rbrace$ the signature of the permutation: $1$ if $\alpha \in S_n$ is an even permutation of the $n$ elements and otherwise $-1$. For $f: \Sigma(k) \rightarrow E$ define its alternation point-wise as the linear idempotent map $$ \mathrm{Alt} f(\sigma) = \dfrac{1}{(k+1)!} \sum_{\alpha \in S_{k+1}} \mathrm{sign}(\alpha) \alpha^*f(\sigma),$$ where $\alpha^*f(\sigma) = f(v_{\alpha(0)}, \dots, v_{\alpha(k)})$ for $\sigma = (v_0, \dots, v_k) \in \Sigma(k).$ As usual, we say that $f$ is alternating if $\mathrm{Alt} f = f$, and symmetric if $\mathrm{Alt} f = 0$.
\end{definition}

Replacing inner product with dual pairing and unitary representations by isometric representations, we next introduce twising and cochains as in \cite{BS}.




\begin{definition}
Let $ \mathcal{E}^{(k,p)}(X,E)$ denote the semi-normed vector space of $k$-cochains $f \colon \Sigma(k) \rightarrow E $ for which the semi norm given by $$\Vert f \Vert_{(k,p)} = \left( \sum_{\sigma \in \Sigma(k, \Gamma)} \Vert f(\sigma) \Vert^p_E \dfrac{\omega(\sigma)}{(k+1)! \vert \Gamma_\sigma\vert}\right)^{1/p},$$ is finite.
\end{definition}



\begin{definition}
For $f \in {\mathcal{E}^{(k,p)}}(X,E)^*$, we denote by $$\langle \phi, f \rangle_k = \sum_{\sigma \in \Sigma(k, \Gamma)} \langle \phi(\sigma), f(\sigma) \rangle_E \dfrac{\omega(\sigma)}{(k+1)! \vert \Gamma_\sigma \vert}$$ the dual pairing between $\mathcal{E}^{(k,p)}(X,E)$ and ${\mathcal{E}^{(k,p)}}(X,E)^*$.
\end{definition}

\begin{proposition} \label{cochaindual}
${\mathcal{E}^{(k,p)}}(X,E)^*\cong \mathcal{E}^{(k,p^*)}(X,E^*)$. \qed 
\end{proposition}   

\begin{definition}
Let $f: \Sigma(k) \rightarrow E$. If for every $g \in \Gamma$ and every $\sigma \in \Sigma(k)$ $$f(g \cdot \sigma) = \pi_g \cdot f(\sigma),$$ then we say that $f$ is twisted by $\pi$, or for short just twisted. 
\end{definition}

\begin{definition}
Let $C^{(k,p)}(X,E)$ denote the vector space of all alternating maps $f \colon \Sigma(k) \rightarrow E$ twisted by $\pi$.
\end{definition}
Those alternating maps twisted by $\pi$ whose $\Vert \cdot \Vert_{(k,p)}$ norm is finite are called $p$-integrable mod $\Gamma$ and we use the following notation:
\begin{definition}
Let $L^{(k,p)}(X, E) = \lbrace f \in C^{(k,p)}(X,E) \colon \Vert f\Vert_{(k,p)} < \infty \rbrace$ denote the vector subspace of all alternating $k$-cochains of $X$ twisted by $\pi$.  
\end{definition}

In particular, if $\Gamma$ acts cocompactly on $X$, then $L^{(k,p)}(X,E) = C^{(k,p)}(X,E)$ since then $X/\Gamma$ is compact, the set of representatives $\Sigma(k, \Gamma)$ is finite, and $\Vert f \Vert_{(k,p)} < \infty$ for all $f \in C^{(k,p)}(X,E)$. \\ We end this section by proving that $L^{(k,p)}(X, E)$ is a normed space with respect to $\Vert \cdot \Vert_{(k,p)}$. Towards this end we first show that $\Vert \cdot \Vert_{(k,p)}$ is independent of the set of representatives when $f \in L^{(k,p)}(X,E)$.

\begin{lemma} \label{independentofrepresentatives}
If $f \in L^{(k,p)}(X, E)$, then $\Vert f \Vert_{(k,p)}$ is independent of the choice of $\Sigma(k, \Gamma)$.
\end{lemma}
\begin{proof} Let $\Sigma'(k, \Gamma)$ be another set of representatives. Then,
\begin{align} \Vert f \Vert_{(k,p)}^p &= \sum_{\sigma' \in \Sigma'(k, \Gamma)} \Vert f(\sigma') \Vert^p_E \dfrac{\omega (\sigma')}{(k+1)!\vert \Gamma_{\sigma'} \vert} = \sum_{\sigma' \in \Sigma'(k, \Gamma)} \Vert f(g' \cdot \sigma') \Vert^p_E \dfrac{\omega (g' \cdot \sigma')}{(k+1)!\vert \Gamma_{g' \cdot \sigma'} \vert} \nonumber \\ &= \sum_{\sigma \in \Sigma(k, \Gamma)} \Vert f(\sigma) \Vert^p_E \dfrac{\omega (\sigma)}{(k+1)!\vert \Gamma_{\sigma} \vert}, \nonumber \end{align}
choosing for each $\sigma' \in \Sigma'(k, \Gamma)$ a $g' \in \Gamma$ such that $g' \cdot \sigma' = \sigma \in \Sigma(k, \Gamma)$ and observing that $f$ is twisted by $\pi$ and both the norm and $\omega$ are $\Gamma$-invariant.
\end{proof}

\begin{proposition}
$L^{(k,p)}(X,E) \subseteq \mathcal{E}^{(k,p)}(X,E)$ is a normed vector space.
\end{proposition} 
\begin{proof} It suffices to show that the seminorm $\Vert \cdot \Vert_{(k,p)}$ on $\mathcal{E}^{(k,p)}(X,E)$ restricted to $L^{(k,p)}(X,E)$ is a norm. To this end, suppose $\Vert f \Vert_{(k,p)} = 0$ for $f \in L^{(k,p)}(X,E)$. By Lemma \ref{independentofrepresentatives} we may assume $f(\sigma) = 0$ for all $\sigma \in \Sigma(k, \Gamma)$. Since $f(g \cdot \sigma) = \pi_g f(\sigma)$ and the action of $\Gamma$ is transitive on the orbits it follows that $f(\sigma)=0$ for all $\sigma \in \Sigma(k)$. \end{proof}


\section{Projecting $k$-cochains onto $L^{(k,p)}(X,E)$} \label{projsec}
In order to extend the framework of \cite{BS}, the dual space of the alternating and twisted cochains has at first to be identified up to isometric isomorphism. Following the scheme presented in \cite{Nowak}, we begin by stepwise constructing a continuous projection $P_L$ from $\mathcal{E}^{(k,p)}(X,E)$ onto $L^{(k,p)}(X,E)$.
\begin{definition} \label{projection1}
Define $\widetilde{P}: \mathcal{E}^{(k,p)}(X,E) \rightarrow \mathcal{E}^{(k,p)}(X,E)$ by \begin{displaymath}
\widetilde{P}f (\sigma) = \left\{ \begin{array}{ll}
\displaystyle \sum_{s \in \Gamma_\sigma} \pi_s f'(\sigma) & \textrm{if $\sigma \in \Sigma(k, \Gamma)$}\\
\displaystyle \sum_{\substack{h \in \Gamma \\ h \cdot \tau = \sigma}} \pi_{h} f'(\tau) & \textrm{if $\sigma \notin \Sigma(k, \Gamma)$ for $\tau \in \Sigma(k, \Gamma)$},\\
\end{array} \right.
\end{displaymath}
where $f': \Sigma(k, \Gamma) \rightarrow E$ is the restriction of $f: \Sigma(k) \rightarrow E$ to $\Sigma(k, \Gamma)$.
\end{definition}
This map is well defined, in particular we note that $\lbrace h \in \Gamma \colon h \cdot \tau = \sigma \rbrace = h\Gamma_\tau$. As the following proposition shows, $\widetilde{P}$ maps $k$-cochains to $k$-cochains twisted by $\pi$.

\begin{proposition}
For $f \in \mathcal{E}^{(k,p)}(X,E)$, the $k$-cochain $\widetilde{P}f: \Sigma(k) \rightarrow E$ is twisted by $\pi$.
\end{proposition}
\begin{proof} Let $\sigma \in \Sigma(k)$. Then either $\sigma \in \Sigma(k, \Gamma)$ or $\sigma \notin \Sigma(k, \Gamma)$. Suppose at first $\sigma \in \Sigma(k, \Gamma)$. If $g \in \Gamma_\sigma$, then clearly $\pi_g \widetilde{P}f(\sigma) =  \widetilde{P}f(g \cdot \sigma)$. On the other hand, if $g \notin \Gamma_\sigma$ we get $$\widetilde{P}f(g \cdot \sigma) = \sum_{\substack{h \in \Gamma \\ h \cdot \sigma = g \cdot \sigma}} \pi_hf'(\sigma),$$
and $$ \pi_g\widetilde{P}f(\sigma) = \sum_{h \in g \Gamma_\sigma} \pi_h f'(\sigma).$$
But $\lbrace h \in \Gamma \colon h \cdot \sigma = g \cdot \sigma \rbrace = \lbrace h \in \Gamma \colon h \in g \Gamma_\sigma \rbrace$, so the claim holds for $\sigma \in \Sigma(k, \Gamma)$.  
Suppose $\sigma \notin \Sigma(k, \Gamma)$. If $g \cdot \sigma \in \Sigma(k, \Gamma)$, then
\begin{align}
\widetilde{P}f(g \cdot \sigma) &= \sum_{s \in \Gamma_{g \cdot \sigma}} \pi_s f'(g \cdot \sigma) = \sum_{s \in g\Gamma_{\sigma}g^{-1}} \pi_s f'(g \cdot \sigma) = \sum_{h \in \Gamma_\sigma} \pi_{ghg^{-1}}f'(g \cdot \sigma) \nonumber, 
\end{align} as $\Gamma_{g \cdot \sigma} = g \Gamma_\sigma g^{-1}$, and so
\begin{align}
\pi_g\widetilde{P}f(\sigma) &= \sum_{\substack{h \in \Gamma \\ hg \cdot \sigma = \sigma}} \pi_{gh}f'(g \cdot \sigma) = \sum_{\substack{h \in \Gamma \\ hg \in \Gamma_\sigma}} \pi_{gh}f'(g \cdot \sigma) = \sum_{h \in \Gamma_\sigma g^{-1}} \pi_{gh}f'(g \cdot \sigma) \nonumber \\ &= \sum_{h \in \Gamma_\sigma} \pi_{ghg^{-1}}f'(g \cdot \sigma) = \widetilde{P}f(g \cdot \sigma). \nonumber 
\end{align}
On the other hand, if $g \cdot \sigma \notin \Sigma(k, \Gamma)$, write $\widetilde{P}f(g \cdot \sigma) = \sum_{h \in A} \pi_h f'(\tau)$ where $A = \lbrace h \in \Gamma \colon h \cdot \tau = g \cdot \sigma \rbrace = gB$ for $B = \lbrace s \in \Gamma \colon s \cdot \tau = \sigma \rbrace$ and $\tau \in \Sigma(k, \Gamma)$. Hence,
\begin{align}
\pi_g \widetilde{P}f(\sigma) &= \sum_{\substack{s \in \Gamma \\ s \cdot \tau = \sigma}} \pi_{gs}f'(\tau) = \sum_{h \in g \lbrace s \in \Gamma \colon s \cdot \tau = \sigma \rbrace} \pi_h f'(\tau)= \sum_{h \in gB} \pi_h f'(\tau) \nonumber \\ &= \sum_{h \in A} \pi_h f'(\tau) = \widetilde{P}f(g \cdot \sigma), \nonumber
\end{align} so the claim holds for $\sigma \notin \Sigma(k, \Gamma)$ as well. \end{proof}
Recalling that $\Gamma$ is discrete, normalizing $\widetilde{P}$ as below gives a projection onto the twisted cochains.
\begin{definition} \label{projection2}
Define ${P}: \mathcal{E}^{(k,p)}(X,E) \rightarrow \mathcal{E}^{(k,p)}(X,E)$ by $$Pf(\sigma) = \dfrac{1}{\vert \Gamma_\sigma \vert}\widetilde{P}f(\sigma).$$ 
\end{definition}

\begin{proposition} \label{projection11}
${P}$ is a projection onto the twisted cochains.
\end{proposition}
\begin{proof} Clearly ${P}^2 = {P}$ and onto. Now, suppose $f$ is twisted and $\sigma \in \Sigma(k)$. If $\sigma \in \Sigma(k, \Gamma)$, then, recalling the discreteness assumption \begin{align}{P}f(\sigma) &= \dfrac{1}{\vert \Gamma_\sigma \vert}\sum_{s \in \Gamma_\sigma} \pi_sf'(\sigma)= \dfrac{1}{\vert \Gamma_\sigma \vert}\sum_{s \in \Gamma_\sigma} f'(s \cdot \sigma) \nonumber = \dfrac{1}{\vert \Gamma_\sigma \vert}\sum_{s \in \Gamma_\sigma} f'(\sigma) = f'(\sigma) =f(\sigma). \end{align} Similarly, for $\sigma \notin \Sigma(k, \Gamma)$
\begin{align}
{P}f(\sigma) &= \dfrac{1}{\vert \Gamma_\sigma \vert}\sum_{\substack{h \in \Gamma \\ h \cdot \tau = \sigma}} \pi_{h} f'(\tau) = \dfrac{1}{\vert \Gamma_\sigma \vert}\sum_{\substack{h \in \Gamma \\ h \cdot \tau = \sigma}} f(h \cdot \tau) = \dfrac{1}{\vert \Gamma_\sigma \vert}\sum_{\substack{h \in \Gamma \\ h \cdot \tau = \sigma}} f(\sigma) = f(\sigma), \nonumber
\end{align} as $\vert \lbrace h \in \Gamma \colon h \cdot \tau = \sigma \rbrace \vert = \vert h \Gamma_\tau \vert = \vert \Gamma_\tau \vert$ and $\vert \Gamma_\sigma \vert = \vert h \Gamma_\tau h^{-1}\vert = \vert \Gamma_\tau \vert$.\end{proof}

\begin{corollary}
$P$ is continuous with $\Vert {P}f \Vert_{(k,p)} \leq \Vert f \Vert_{(k,p)}$ for $f \in \mathcal{E}^{(k,p)}(X,E)$ with equality for $f \in L^{(k,p)}(X,E)$. 
\end{corollary}
\begin{proof} A straightforward consequence of Definition \ref{projection2}, and the observation that ${P}f=f$ for $f \in L^{(k,p)}(X,E)$.\end{proof}
Thus, we have constructed a projection $P$ onto the cochains twisted by $\pi$. However, cochains in the image are not necessarily alternating and hence not necessarily in $L^{(k,p)}(X,E)$. Antisymmetrizing $P$ fixes this. We begin with the following observation:
\begin{corollary} \label{remainstwisted} \label{propAlt}
If $f$ is twisted, then $\mathrm{Alt} f$ is twisted.
\end{corollary}
\begin{proof} Suppose $f: \Sigma(k) \rightarrow E$ is twisted. Then, $\mathrm{Alt}f(g \cdot \sigma) = \mathrm{Alt}(\pi_g f(\sigma)) = \pi_g (\mathrm{Alt}f(\sigma))$ for all $g \in \Gamma$ and $\sigma \in \Sigma(k)$, where we used the fact that $f$ is twisted in the first equality and linearity of $\pi_g$ in the last equality. Hence, $\mathrm{Alt}f$ is twisted as well. \end{proof}
\begin{corollary} \label{althelp}
Suppose $f \in \mathcal{E}^{(k,p)}(X,E)$, then $$\Vert \mathrm{Alt}f \Vert_{(k,p)}^p \leq (k+1)! \Vert f \Vert_{(k,p)}^p$$
\end{corollary}
\begin{proof} Since
\begin{align} \Vert \mathrm{Alt}f(\sigma) \Vert^p_E &= \dfrac{1}{(k+1)!^p} \Vert \sum_{\alpha \in S_{k+1}} \mathrm{sign}(\alpha)\alpha^*f(\sigma) \Vert^p_E \leq \dfrac{(k+1)!^p}{(k+1)!^p} \sum_{\alpha \in S_{k+1}}  \Vert \alpha^*f(\sigma) \Vert^p_E \nonumber \\ &= \sum_{\alpha \in S_{k+1}}\Vert \alpha^*f(\sigma) \Vert^p_E, \nonumber \end{align} it follows that \begin{align}\Vert \mathrm{Alt}f \Vert^p_{(k,p)} &= \sum_{\sigma \in \Sigma(k,\Gamma)} \Vert \mathrm{Alt} f(\sigma) \Vert^p \dfrac{\omega(\sigma)}{(k+1)!\vert \Gamma_\sigma \vert} \nonumber \\ &\leq \sum_{\alpha \in S_{k+1}} \sum_{\sigma \in \Sigma(k,\Gamma)} \Vert \alpha^* f(\sigma) \Vert^p \dfrac{\omega(\sigma)}{(k+1)!\vert \Gamma_\sigma \vert} \nonumber \\ &= (k+1)! \sum_{\sigma \in \Sigma(k,\Gamma)} \Vert  f(\sigma) \Vert^p \dfrac{\omega(\sigma)}{(k+1)!\vert \Gamma_\sigma \vert} \nonumber \end{align}
since we sum over all representatives in the last equality, and for $\sigma = (v_0, \dots, v_k)$, $\omega(v_{\alpha(0)}, \dots, v_{\alpha(k)}) = \omega(\sigma)$ and $\Gamma_{(v_{\alpha(0)}, \dots v_{\alpha(k)})} = \Gamma_\sigma$ for all $\alpha \in S_{k+1}$. \end{proof}

\begin{proposition}
The map $P_L : \mathcal{E}^{(k,p)}(X,E) \rightarrow \mathcal{E}^{(k,p)}(X,E)$, given by $$P_L= \mathrm{Alt} \circ {P}$$ defines a projection onto $L^{(k,p)}(X,E)$. In other words, the diagram $$
\xymatrix{
{\mathcal{E}^{(k,p)}(X,E)} \ar@{->>}[r]^P \ar@{->>}[dr]_{{P_L}} & 
{\lbrace \mathrm{k-cochain}_{\pi}\rbrace} \ar@{->>}[d]^{{\mathrm{Alt}}} \\ {} & {{L^{(k,p)}(X,E)}}
} $$ commutes. 
\end{proposition}
\begin{proof} Clearly $P_L^2 = P_L$. By Proposition \ref{projection11} $P$ is a projection onto the twisted cochains, and since $\mathrm{Alt}$ preserves twisting by Corollary \ref{remainstwisted}, $P_L$ is a projection onto $L^{(k,p)}(X,E)$.\end{proof}

\begin{proposition} \label{P_Lcontinuous}
$P_L$ is continuous with $\Vert P_Lf \Vert_{(k,p)}^p \leq (k+1)! \Vert f \Vert_{(k,p)}^p$. 
\end{proposition}
\begin{proof} \begin{align}
\Vert P_Lf \Vert_{(k,p)}^p &= \sum_{\sigma \in \Sigma(k, \Gamma)} \Vert P_Lf(\sigma) \Vert^p_E \dfrac{\omega(\sigma)}{(k+1)!\vert \Gamma_\sigma \vert} \nonumber \\ 
&= \sum_{\sigma \in \Sigma(k, \Gamma)} \left\Vert \mathrm{Alt}\left( \dfrac{1}{\vert \Gamma_\sigma \vert}\sum_{s \in \Gamma_\sigma} \pi_s f'(\sigma) \right)\right\Vert^p_E \dfrac{\omega(\sigma)}{(k+1)!\vert \Gamma_\sigma \vert} \nonumber \\ &\substack{{} \\ =} \sum_{\sigma \in \Sigma(k, \Gamma)} \left\Vert \dfrac{1}{\vert \Gamma_\sigma \vert}\sum_{s \in \Gamma_\sigma} \pi_s \mathrm{Alt} f'(\sigma) \right\Vert^p_E \dfrac{\omega(\sigma)}{(k+1)!\vert \Gamma_\sigma \vert} \nonumber \\ & \leq \sum_{\sigma \in \Sigma(k, \Gamma)} \dfrac{1}{\vert \Gamma_\sigma \vert^p} \vert \Gamma_\sigma \vert^p \max_{s \in \Gamma_s} \left\lbrace \Vert \pi_s \mathrm{Alt}f'(\sigma)\Vert^p_E \right\rbrace \dfrac{\omega(\sigma)}{(k+1)!\vert \Gamma_\sigma \vert} \nonumber \\ &= \sum_{\sigma \in \Sigma(k, \Gamma)} \Vert \mathrm{Alt}f'(\sigma)\Vert^p_E \dfrac{\omega(\sigma)}{(k+1)!\vert \Gamma_\sigma \vert} \nonumber \\ &\substack{{} \\ \leq} \sum_{\sigma \in \Sigma(k, \Gamma)} (k+1)! \Vert f'(\sigma)\Vert^p_E \dfrac{\omega(\sigma)}{(k+1)!\vert \Gamma_\sigma \vert} \nonumber \\ &= (k+1)! \Vert f \Vert_{(k,p)}^p, \nonumber
\end{align}
where we have used Corollary \ref{althelp} in the last inequality.

\end{proof}

\begin{corollary} \label{Lclosed}
$\displaystyle L^{(k,p)}(X,E) \subseteq \mathcal{E}^{(k,p)}(X,E)$ is closed. \qed
\end{corollary} 

\section{$L^{(k,p)}(X,E)^* \cong L^{(k,p^*)}(X,E^*)$} \label{L}
Having constructed a continuous projection from $\mathcal{E}^{(k,p)}(X,E)$ onto $L^{(k,p)}(X,E)$ we show that the dual of $L^{(k,p)}(X,E)$ can be identified up to isometric isomorphism with $\mathcal{E}^{(k,p^*)}(X,E^*) / \mathrm{Ann}(L^{(k,p)}(X,E))$, cf. Corollary \ref{fin2}, and finally that the latter is isometrically isomorphic to $L^{(k,p^*)}(X,E^*)$, cf. Proposition \ref{isometric1} and \ref{fin4} below. Towards this end, recall that by the annihilator of a subspace $M \subseteq E$ we mean the vector space $\mathrm{Ann}(M) = \lbrace x \in E^* \colon \langle y,x \rangle_E = 0 \,\, \forall y \in M \rbrace$ of all bounded linear functionals on $E$ that vanish on $M$. The following fact contains the idea of the proof:

\begin{proposition} \label{dualcomposition1} \label{firstdual}
\cite{Douglas} Suppose $E$ is a Banach space such that $E= M \oplus N$ and denote by $P$ the corresponding projection onto $M$. Then, 
\begin{enumerate}
\item $\ker P^* = \mathrm{Ann}(M)$ and $\mathop{\mathrm{im}}P^* = \mathrm{Ann}(N)$;
\item $E^* \simeq \mathrm{Ann}(N) \oplus \mathrm{Ann}(M)$;
\item if $M$ is closed $M^* \cong E^* / \mathrm{Ann}(M).$ \qed
\end{enumerate} 
\end{proposition} 
Let $L^{(k,p)}_-(X,E)$ denote the closed complement of $L^{(k,p)}(X,E)$ in $\mathcal{E}^{(k,p)}(X,E)$. That is $ L^{(k,p)}_-(X,E) = \ker \, {P}_{L}$, or in other words:

\begin{corollary} \label{symmetry} \label{fin1}
$L_-^{(k,p)}(X,E) = \lbrace f \in \mathcal{E}^{(k,p)}(X,E) \colon \mathrm{Alt}f(\sigma) = 0 \,\, \forall \sigma \in \Sigma(k, \Gamma) \rbrace$ is a closed subspace of $\mathcal{E}^{(k,p)}(X,E).$
\end{corollary} \begin{proof} Given $f \in L^{(k,p)}_-(X,E)$, $(I-P_L)f(\sigma) = f(\sigma)$ for all $\sigma \in \Sigma(k)$, and hence for all $\sigma \in \Sigma(k, \Gamma)$
\begin{align}
(I-P_L)f'(\sigma) &= f'(\sigma) - P_Lf'(\sigma) = f'(\sigma) - \dfrac{1}{\vert \Gamma_\sigma \vert} \sum_{s \in \Gamma_\sigma} \pi_s \mathrm{Alt}f'(\sigma) = f'(\sigma), \nonumber  
\end{align}
implying by linearity that $\mathrm{Alt}f'(\sigma) = 0$. Hence, $f$ is symmetric on representatives.
\end{proof}

\begin{corollary} \label{fin2}
$\displaystyle L^{(k,p)}(X,E)^* \cong \mathcal{E}^{(k,p^*)}(X,E^*) / \mathrm{Ann}(L^{(k,p)}(X,E)).$
\end{corollary}
\begin{proof} Since $L^{(k,p)}(X,E)$ is a closed subspace of $\mathcal{E}^{(k,p)}(X,E)$ by Corollary \ref{Lclosed}, the claim now follows from Proposition \ref{firstdual}(3) and the fact that ${\mathcal{E}^{(k,p)}(X,E)}^* \cong \mathcal{E}^{(k,p^*)}(X,E^*)$. \end{proof}
It now remains to identify the annihilators, cf. Proposition \ref{isometric1}, to prove isomorphism and finally isometry. As indicated by Proposition \ref{dualcomposition1} this requires knowing $P_L^*$.
\begin{proposition} \label{P_Ladjoint}
Let $\overline{P}_L : \mathcal{E}^{(k,p^*)}(X,E^*) \rightarrow \mathcal{E}^{(k,p^*)}(X,E^*)$ be a projection as above. Then $\overline{P}_L = P_L^*$.
\end{proposition}
\begin{proof} Assume first $k=1$, let $f \in \mathcal{E}^{(1,p)}(X,E)$ and $\phi \in \mathcal{E}^{(1,p^*)}(X,E^*)$. For $\sigma = (v_0, v_1) \in \Sigma(1, \Gamma)$ we denote by $-\sigma$ the simplex $(v_1,v_0)$. Now, 
\begin{align}
&\langle P_Lf, \phi \rangle_1 = \sum_{\sigma \in \Sigma(1, \Gamma)} \langle P_Lf(\sigma), \phi(\sigma) \rangle_E \dfrac{\omega(\sigma)}{2! \vert \Gamma_\sigma \vert} \nonumber \\ &\substack{(*) \\ =} \sum_{\sigma \in \Sigma(1, \Gamma)} \left\langle \dfrac{1}{\vert \Gamma_\sigma \vert} \sum_{s \in \Gamma_s} \pi_s \left( \dfrac{1}{2} f'(\sigma) - \dfrac{1}{2}f'(-\sigma)\right), \phi'(\sigma)\right\rangle_E \dfrac{\omega(\sigma)}{2! \vert \Gamma_\sigma \vert} \nonumber \\ &= \sum_{\sigma \in \Sigma(1, \Gamma)} \dfrac{1}{2 \vert \Gamma_\sigma \vert}\sum_{s \in \Gamma_\sigma} \left\langle \pi_s \left(f'(\sigma) - f'(-\sigma)\right), \phi'(\sigma)\right\rangle_E \dfrac{\omega(\sigma)}{2! \vert \Gamma_\sigma \vert} \nonumber \\ &= \sum_{\sigma \in \Sigma(1, \Gamma)} \dfrac{1}{2 \vert \Gamma_\sigma \vert} \sum_{s \in \Gamma_\sigma} \left\langle \pi_s f'(\sigma), \phi'(\sigma) \right\rangle_E \dfrac{\omega(\sigma)}{2! \vert \Gamma_\sigma \vert} \nonumber \\ &\phantom{=} -\sum_{\sigma \in \Sigma(1, \Gamma)} \dfrac{1}{2 \vert \Gamma_\sigma \vert}\sum_{s \in \Gamma_\sigma} \left\langle \pi_s f'(-\sigma), \phi'(\sigma)\right\rangle_E \dfrac{\omega(\sigma)}{2! \vert \Gamma_\sigma \vert} \nonumber \\ &\substack{(**) \\ =} \sum_{\sigma \in \Sigma(1, \Gamma)} \dfrac{1}{2 \vert \Gamma_\sigma \vert} \sum_{s \in \Gamma_\sigma} \left\langle \pi_s f'(\sigma), \phi'(\sigma) \right\rangle_E \dfrac{\omega(\sigma)}{2! \vert \Gamma_\sigma \vert} \nonumber \\ &\phantom{=} -\sum_{\sigma \in \Sigma(1, \Gamma)} \dfrac{1}{2 \vert \Gamma_\sigma \vert}\sum_{s \in \Gamma_\sigma} \left\langle \pi_s f'(\sigma), \phi'(-\sigma)\right\rangle_E 
\dfrac{\omega(\sigma)}{2! \vert \Gamma_\sigma \vert} \nonumber \end{align} \begin{align} &= \sum_{\sigma \in \Sigma(1, \Gamma)} \dfrac{1}{2 \vert \Gamma_\sigma \vert} \sum_{s \in \Gamma_\sigma} \left\langle f'(\sigma), \overline{\pi}_s \phi'(\sigma) \right\rangle_E \dfrac{\omega(\sigma)}{2! \vert \Gamma_\sigma \vert} \nonumber \\ &\phantom{=} -\sum_{\sigma \in \Sigma(1, \Gamma)} \dfrac{1}{2 \vert \Gamma_\sigma \vert} \sum_{s \in \Gamma_\sigma} \left\langle f'(\sigma), \overline{\pi}_s\phi'(-\sigma)\right\rangle_E \dfrac{\omega(\sigma)}{2! \vert \Gamma_\sigma \vert} \nonumber \\ &= \sum_{\sigma \in \Sigma(1, \Gamma)} \dfrac{1}{2 \vert \Gamma_\sigma \vert} \sum_{s \in \Gamma_\sigma} \left\langle f'(\sigma), \overline{\pi}_s \left(\phi'(\sigma)-\phi'(-\sigma) \right) \right\rangle_E \dfrac{\omega(\sigma)}{2! \vert \Gamma_\sigma \vert} \nonumber \\ &= \langle f, \overline{P}_L \phi \rangle_1,\nonumber 
\end{align} where $(*)$ and the last equality follow from the definition of $P_L$ and $\overline{P}_L$, respectively when $k=1$. $(**)$ follows as we sum over all $\sigma \in \Sigma(1, \Gamma)$, so the sums where we switch the summation variable $\sigma$ with $-\sigma$ agree as $\omega(\sigma) = \omega(-\sigma)$. For $k>1$ the calculation goes similarly, denoting $\sigma = (v_0, \dots, v_k) \in \Sigma(k, \Gamma)$ and arguing similarly,
\begin{align}
&\langle P_Lf, \phi \rangle_k = \sum_{\sigma \in \Sigma(k, \Gamma)} \langle P_Lf(\sigma), \phi(\sigma) \rangle_E \dfrac{\omega(\sigma)}{(k+1)! \vert \Gamma_\sigma \vert} \nonumber \\ &\substack{{} \\ =}\sum_{\sigma \in \Sigma(k, \Gamma)} \sum_{s \in \Gamma_s} \sum_{\alpha \in S_{k+1}} \dfrac{1}{\vert \Gamma_\sigma \vert} \dfrac{1}{(k+1)!}  (-1)^{\mathrm{sgn}(\alpha)} \left\langle \pi_s  f'((v_{\alpha(0)}, \dots, v_{\alpha(k)})), \phi'(\sigma)\right\rangle_E \dfrac{\omega(\sigma)}{(k+1)! \vert \Gamma_\sigma \vert} \nonumber \\ &\substack{{} \\ =}\sum_{\sigma \in \Sigma(k, \Gamma)} \sum_{s \in \Gamma_s} \sum_{\alpha \in S_{k+1}} \dfrac{1}{\vert \Gamma_\sigma \vert} \dfrac{1}{(k+1)!}  (-1)^{\mathrm{sgn}(\alpha)} \left\langle \pi_s  f'(\sigma), \phi'((v_{\alpha(0)}, \dots, v_{\alpha(k)}))\right\rangle_E \dfrac{\omega(\sigma)}{(k+1)! \vert \Gamma_\sigma \vert} \nonumber \\ &=\sum_{\sigma \in \Sigma(k, \Gamma)} \sum_{s \in \Gamma_s} \sum_{\alpha \in S_{k+1}} \dfrac{1}{\vert \Gamma_\sigma \vert} \dfrac{1}{(k+1)!}  (-1)^{\mathrm{sgn}(\alpha)} \left\langle  f'(\sigma), \overline{\pi}_s \phi'((v_{\alpha(0)}, \dots, v_{\alpha(k)}))\right\rangle_E \dfrac{\omega(\sigma)}{(k+1)! \vert \Gamma_\sigma \vert} \nonumber \\ &= \langle f, \overline{P}_L \phi \rangle_k, \nonumber
\end{align} 
where the first and last equality follows by the definition of $P_L$ and linearity of the dual pairing, and the third similarly as in the case $k=1$ above.\end{proof}

\begin{proposition} \label{isometric1}
The following are equal as sets:
\begin{enumerate}
\item $\mathrm{Ann}(L_-^{(k,p)}(X,E)) = L^{(k,p^*)}(X,E^*);$ 
\item $\mathrm{Ann}(L^{(k,p)}(X,E)) = L_-^{(k,p^*)}(X,E^*).$
\end{enumerate}
\end{proposition}
\begin{proof} Suppose $f \in L_-^{(k,p)}(X,E)$ and $\phi \in \mathcal{E}^{(k,p^*)}(X,E^*)$. Then,
\begin{align}
\langle f, \phi \rangle_k &= \sum_{\sigma \in \Sigma(k, \Gamma)} \langle f(\sigma), \phi(\sigma) \rangle_E \dfrac{\omega(\sigma)}{(k+1)!\vert \Gamma_\sigma \vert} \nonumber \\ &= \sum_{\sigma \in \Sigma(k, \Gamma)} \langle (I-P_L)f(\sigma), \phi(\sigma) \rangle_E \dfrac{\omega(\sigma)}{(k+1)!\vert \Gamma_\sigma \vert} \nonumber \\ &= \langle f,\phi \rangle_k - \langle P_L f,\phi \rangle_k = \langle f,\phi \rangle_k - \langle f,\overline{P}_L \phi \rangle_k \nonumber.
\end{align} Hence, $\langle f, \overline{P}_L \phi \rangle_E = 0$ for all $f \in L_-^{(k,p)}(X,E)$ where $\overline{P}_L \phi \in L^{(k,p^*)}(X,E^*)$. Thus, $L^{(k,p^*)}(X,E^*) \subseteq \mathrm{Ann}(L_-^{(k,p)}(X,E))$. On the other hand, suppose $\phi \in \mathrm{Ann}(L_-^{(k,p)}(X,E))$, then $\langle f, \phi \rangle_k = 0$ for all $f \in L_-^{(k,p)}(X,E)$. Hence, for all $f \in \mathcal{E}^{(k,p^*)}(X,E^*)$, so $0= \langle (I-P_L)f, \phi \rangle_k$ if and only if $\langle f, \phi \rangle_k = \langle P_Lf, \phi \rangle_k = \langle f, \overline{P}_L \phi \rangle_k$ for all $f \in \mathcal{E}^{(k,p^*)}(X,E^*)$. Thus, $\phi = \overline{P}_L \phi$ implies that $\phi \in L^{(k,p^*)}(X,E^*)$, so $\mathrm{Ann}(L_-^{(k,p)}(X,E)) \subseteq L^{(k,p^*)}(X,E^*)$ proving the first claim. The proof of the second claim goes similarly. \end{proof}

\begin{corollary} \label{isometry1}
The following are isomorphic:
\begin{enumerate}
\item $ L^{(k,p)}(X,E)^* \cong \mathcal{E}^{(k,p^*)}(X,E^*) \, / L^{(k,p^*)}_-(X,E^*) \simeq L^{(k,p^*)}(X,E^*);$
\item $ L^{(k,p)}_-(X,E)^* \cong \mathcal{E}^{(k,p^*)}(X,E^*) \, / L^{(k,p^*)}(X,E^*) \simeq L^{(k,p^*)}_-(X,E^*).$
\end{enumerate}
\end{corollary}
\begin{proof} The first isomorphic isomorphisms follow immediately combining Propositions \ref{firstdual} and \ref{isometric1}, and the latter isomorphisms by Proposition \ref{dualcomposition1}.\end{proof}

\begin{proposition} \label{fin4}
The following are isometrically isomorphic:
\begin{enumerate}
\item $L^{(k,p)}(X,E)^* \cong L^{(k,p^*)}(X,E^*);$
\item $L^{(k,p)}_-(X,E)^* \cong L^{(k,p^*)}_-(X,E^*).$
\end{enumerate}
\end{proposition}
Proof. Consider the second claim. Consider $\mathcal{E}^{(k,p^*)}(X,E^*) \, / L^{(k,p^*)}(X,E^*)$ consisting of the cosets $[\phi] = \phi + L^{(k,p^*)}(X,E^*)$ for $\phi \in \mathcal{E}^{(k,p^*)}(X,E^*)$. We claim that if $\phi \in L^{(k,p^*)}_-(X,E^*)$, then $\Vert [\phi]\Vert = \Vert \phi \Vert_{(k,p^*)}$ where $\Vert \cdot \Vert = \inf \lbrace \Vert \xi \Vert_ {(k,p^*)} \colon \xi \in [\phi] \rbrace$ is the quotient norm. On the other hand, $\mathcal{E}^{(k,p^*)}(X,E^*) \, / L^{(k,p^*)}(X,E^*) \simeq L^{(k,p^*)}_-(X,E^*)$ by Corollary \ref{isometry1}, so $\mathcal{E}^{(k,p^*)}(X,E^*) \, / L^{(k,p^*)}(X,E^*) \cong L^{(k,p^*)}_-(X,E^*)$ and consequently $L^{(k,p)}_-(X,E)^* \cong L^{(k,p^*)}_-(X,E^*)$. Towards this end, fix $\phi \in L^{(k,p^*)}_-(X,E^*)$. Thus, for $\psi \in L^{(k,p^*)}(X,E^*)$ we have
\begin{align}
\Vert \phi + \psi \Vert_{(k,p^*)}^{p^*} &= \sum_{\sigma \in \Sigma(k, \Gamma)} \Vert \phi(\sigma) + \psi(\sigma) \Vert^{p^*}_{E^*} \dfrac{\omega(\sigma)}{(k+1)!\vert \Gamma_\sigma \vert} \nonumber \\ &\substack{{} \\ =} \sum_{\sigma \in \Sigma(k, \Gamma)} \Vert \phi(\sigma) - \psi(-\sigma) \Vert^{p^*}_{E^*} \dfrac{\omega(\sigma)}{(k+1)!\vert \Gamma_\sigma \vert} \nonumber \\ &\substack{{} \\ =} \sum_{\sigma \in \Sigma(k, \Gamma)} \Vert \phi(-\sigma) - \psi(-\sigma) \Vert^{p^*}_{E^*} \dfrac{\omega(\sigma)}{(k+1)!\vert \Gamma_\sigma \vert} \nonumber \\ &\substack{{} \\ =} \sum_{\sigma \in \Sigma(k, \Gamma)} \Vert \phi(-\sigma) - \psi(-\sigma) \Vert^{p^*}_{E^*} \dfrac{\omega(-\sigma)}{(k+1)!\vert \Gamma_{-\sigma} \vert} \nonumber \\ &\substack{{} \\ =} \Vert \phi - \psi \Vert^{p^*}_{(k,p^*)}, \nonumber
\end{align} 
where $-\sigma = (v_1,v_0, v_2, \dots, v_k)$ for $\sigma = (v_0, \dots, v_k)$. The second equality follows since $\psi$ is alternating, the third equality since $\phi$ is symmetric on representatives, the fourth since $\omega$ is symmetric and $\Gamma_\sigma = \Gamma_{- \sigma}$, and the last equality holds since we sum over all $\sigma \in \Sigma(k, \Gamma)$, which contains all the oriented simplexes with the vertices of $\sigma$. Thus, $\Vert \phi + \psi \Vert_{(k,p^*)}^{p^*} = \Vert \phi - \psi \Vert_{(k,p^*)}^{p^*}$ and consequently by the triangle inequality $$2\Vert \phi \Vert_{(k,p^*)} = \Vert 2 \phi + \psi - \psi \Vert_{(k,p^*)} \leq \Vert \phi + \psi \Vert_{(k,p^*)} + \Vert \phi - \psi \Vert_{(k,p^*)},$$ and so $$\Vert \phi \Vert_{(k,p^*)} \leq \dfrac{1}{2}(\Vert \phi + \psi \Vert_{(k,p^*)} + \Vert \phi - \psi \Vert_{(k,p^*)}) = \Vert \phi + \psi \Vert_{(k,p)}.$$ Now, taking the infimum over $\psi \in L^{(k,p^*)}(X,E^*)$ thus shows that the quotient norm of $[\phi]$ is $\Vert \phi\Vert_{(k,p^*)}$, proving the second claim. The first claim is proven similarly by considering the cosets in $\mathcal{E}^{(k,p^*)}(X,E^*) \, / L^{(k,p^*)}_-(X,E^*)$. \begin{flushright}$\Box$ \end{flushright}


\section{Differentials and codifferentials} \label{DC}
Having identified the dual of $L^{(k,p)}(X,E)$ up to isometric isomorphism we extend the notion of differentials and codifferentials as presented in \cite{BS} to reflexive Banach spaces.   
\begin{definition}
Codifferentials and differentials. Define the codifferential $$d_k: \mathcal{E}^{(k,p)}(X,E)  \rightarrow  \mathcal{E}^{(k+1,p)}(X,E)$$ point-wise by $$d\phi(\sigma) = \sum_{i=0}^{k+1}(-1)^i\phi(\sigma_i),$$ 
and the differential $$\delta_{k+1}: L^{(k+1,p^*)}(X,E^*) \rightarrow L^{(k,p^*)}(X,E^*),$$ as the adjoint of $d$ given by $\langle \phi , d \psi \rangle_{k+1} = \langle \delta \phi, \psi \rangle_k$ for all $\psi \in L^{(k,p)}(X,E)$ and $\phi \in L^{(k+1,p^*)}(X,E^*)$.
\end{definition}
Similarly, we denote by $\bar{d}_k: \mathcal{E}^{(k,p^*)}(X,E^*) \rightarrow \mathcal{E}^{(k+1,p^*)}(X,E^*)$ the map given by $\bar{d}\psi(\sigma) =  \sum_{i=0}^{k}(-1)^i\psi(\sigma_i)$ for $\psi \in \mathcal{E}^{(k,p^*)}(X,E^*)$, and likewise for the differential.

\begin{corollary}
$$\xymatrix@1{{\cdots} \ar[r]^-{\delta_{k+2}} & L^{(k+1,p^*)}(X,E^*) \ar[r]^-{\delta_{k+1}} & L^{(k,p^*)}(X,E^*) \ar[r]^-{\delta_k} & {\cdots} 
}$$
is a chain complex over $\mathbb{R}$ dual to the cochain complex 
$$\xymatrix@1{{\cdots\,} & L^{(k+1,p)}(X,E) \ar[l]_-{d_{k+1}} & L^{(k,p)}(X,E) \ar[l]_-{d_k} & {\cdots} \ar[l]_-{d_{k-1}} 
}$$
\qed \end{corollary} 
As the following shows, both $d$ and $\delta$ are bounded operators.
\begin{proposition}
Let $\phi \in L^{(k,p)}(X,E)$. Then  $d: L^{(k,p)}(X,E) \rightarrow L^{(k+1,p)}(X,E)$ is a bounded operator and $$\Vert d\phi \Vert^p_{(k+1,p)} \leq (n-k)(k+2)^p\Vert \phi \Vert^p_{(k,p)}.$$ 
\end{proposition}
\begin{proof} Using the estimate \begin{align} \left\Vert \sum_{i=0}^{k+1} (-1)^i \phi(s_i) \right\Vert_E^p &\leq \left( \Vert \phi(s_0) \Vert_E + \left\Vert \sum_{i=1}^{k+1} (-1)^i \phi(s_i) \right\Vert_E \right)^p \leq \left( \sum_{i=0}^{k+1} \Vert \phi(s_i)\Vert_E \right)^p \nonumber \\ &\leq ((k+2) \max \lbrace \Vert \phi(s_0) \Vert, \dots, \Vert \phi(s_{k+1})\Vert \rbrace )^p \nonumber \\ &\leq (k+2)^p \sum_{i=0}^{k+1} \Vert \phi(s_i) \Vert_E^p , \nonumber \end{align} it follows that \begin{align}
\Vert d \phi \Vert^p_{(k+1,p)} &= \sum_{s \in \Sigma(k+1, \Gamma)} \Vert d \phi (s) \Vert_E^p \dfrac{\omega(s)}{(k+2)! \vert \Gamma_s \vert} \nonumber \\ &= \sum_{s \in \Sigma(k+1, \Gamma)} \left\Vert \sum_{i=0}^{k+1}(-1)^i \phi(s_i) \right\Vert_E^p \dfrac{\omega(s)}{(k+2)! \vert \Gamma_s \vert} \nonumber \\ &\substack{{} \\ \leq} \sum_{s \in \Sigma(k+1, \Gamma)} \left( \dfrac{(k+2)^p\omega(s)}{(k+2)! \vert \Gamma_s \vert} \sum_{i=0}^{k+1}\Vert \phi(s_i)\Vert_E^p \right) \nonumber \end{align} \begin{align}  &\substack{{} \\ =} \sum_{s \in \Sigma(k+1, \Gamma)} \left( \dfrac{(k+2)^p\omega(s)}{(k+2)!(k+1)! \vert \Gamma_s \vert} \sum_{\substack{t \in \Sigma(k) \\ t \subset s}} \Vert [s:t] \phi(t)\Vert_E^p \right) \nonumber \\ &\substack{{} \\ =} \sum_{s \in \Sigma(k+1, \Gamma)} \left( \dfrac{(k+2)^p\omega(s)}{(k+2)!(k+1)! \vert \Gamma_s \vert} \sum_{\substack{t \in \Sigma(k) \\ t \subset s}} \Vert \phi(t)\Vert_E^p \right) \nonumber \\ 
&\substack{(*) \\ =}\sum_{t \in \Sigma(k, \Gamma)} \left( \dfrac{(k+2)^p \Vert \phi(t)\Vert^p_E}{(k+2)!(k+1)!\vert \Gamma_t \vert} \sum_{\substack{s \in \Sigma(k+1) \\ t \subset s}} \omega(s) \right) \nonumber \\ 
&\substack{(**) \\ =} \sum_{t \in \Sigma(k, \Gamma)} \dfrac{(n-k)(k+2)!(k+2)^p\omega(t)}{(k+2)!(k+1)! \vert \Gamma_t \vert} \Vert \phi(t)\Vert_E^p \nonumber \\ &= (n-k)(k+2)^p\Vert \phi \Vert^p_{(k,p)}, \nonumber
\end{align}
where we used Proposition \ref{switchingsums} in $(*)$ followed by Proposition \ref{combinatorial} in $(**)$.
\end{proof}


Similarly to \cite{BS}, we also have the following useful point-wise expression for the differential.
\begin{proposition} \label{codifferential}
For $\phi \in L^{(k,p^*)}(X, E^*)$ and $\tau \in \Sigma(k-1)$ $$\delta \phi(\tau) = \sum_{\substack{v \in \Sigma(0) \\ v * \tau \in \Sigma(k)}} \dfrac{\omega(v * \tau)}{\omega(\tau)} \phi(v * \tau).$$ 
\end{proposition}
\begin{proof} Let $\psi \in L^{(k-1,p)}(X,E)$. The claim follows by a straightforward computation,
\begin{align}
\langle \phi, d \psi\rangle_{k} &=  \sum_{\sigma \in \Sigma(k, \Gamma)} \dfrac{\omega(\sigma)}{(k+1)! \vert \Gamma_\sigma\vert}\langle \phi(\sigma), d \psi(\sigma)\rangle_E \nonumber \\ &=  \sum_{\sigma \in \Sigma(k, \Gamma)} \dfrac{\omega(\sigma)}{(k+1)! \vert \Gamma_\sigma\vert} \langle \phi(\sigma), \sum_{i=0}^{k-1}(-1)^i \psi(\sigma_i)\rangle_E \nonumber \\ &= \sum_{\sigma \in \Sigma(k, \Gamma)} \dfrac{\omega(\sigma)}{(k+1)! \vert \Gamma_\sigma\vert} \langle \phi(\sigma), \dfrac{1}{(k-1+1)!} \sum_{\substack{\tau \in \Sigma(k-1) \\ \tau \subset \sigma}} [\sigma: \tau] \psi(\tau)\rangle_E \nonumber \\ &= \sum_{\sigma \in \Sigma(k, \Gamma)} \sum_{\substack{\tau \in \Sigma(k-1) \\ \tau \subset \sigma}} [\sigma :\tau]\dfrac{\omega(\sigma)\langle \phi(\sigma), \psi(\tau)\rangle_E}{(k+1)! \vert \Gamma_\sigma\vert k!} \nonumber 
\nonumber 
\end{align}
\begin{align}
\phantom{\langle \phi, d \psi\rangle_{k} = \,} &\substack{(*) \\ =} \sum_{\tau \in \Sigma(k-1, \Gamma)} \sum_{\substack{\sigma \in \Sigma(k) \\ \tau \subset \sigma}} [\sigma :\tau] \dfrac{\omega(\sigma)\langle \phi(\sigma), \psi(\tau)\rangle_E}{(k+1)! \vert \Gamma_\tau\vert k!} \nonumber \\ &= \sum_{\tau \in \Sigma(k-1, \Gamma)} \dfrac{\omega(\tau)}{\omega(\tau)} \sum_{\substack{\sigma \in \Sigma(k) \\ \tau \subset \sigma}} [\sigma :\tau] \dfrac{\omega(\sigma)\langle \phi(\sigma), \psi(\tau)\rangle_E}{(k+1)! \vert \Gamma_\tau\vert k!} \nonumber \\ &= \sum_{\tau \in \Sigma(k-1, \Gamma)} \dfrac{\omega(\tau)}{k! \vert \Gamma_\tau \vert} \sum_{\substack{\sigma \in \Sigma(k) \\ \tau \subset \sigma}} [\sigma :\tau] \dfrac{\omega(\sigma)\langle \phi(\sigma), \psi(\tau)\rangle_E}{(k+1)! \omega(\tau)} \nonumber \\  &= \sum_{\tau \in \Sigma(k-1, \Gamma)} \dfrac{\omega(\tau)}{k! \vert \Gamma_\tau \vert} \sum_{\substack{v \in \Sigma(0) \\ v * \tau \in \Sigma(k)}} [v\tau : \tau] \dfrac{\omega(v * \tau)}{\omega(\tau)}\langle \phi(v * \tau), \psi(\tau)\rangle_E \nonumber \\ &= \sum_{\tau \in \Sigma(k-1, \Gamma)} \dfrac{\omega(\tau)}{k! \vert \Gamma_\tau \vert} \left\langle \sum_{\substack{v \in \Sigma(0) \\ v * \tau \in \Sigma(k)}} \dfrac{\omega(v * \tau)}{\omega(\tau)} \phi(v * \tau), \psi(\tau) \right\rangle_E = \langle \delta \phi, \psi\rangle_{k-1}, \nonumber
\end{align} where we used Proposition \ref{switchingsums} in $(*)$ above. \end{proof}

\section{Localization and restriction} \label{secloc}
In this section we recall the concept of localization following \cite{BS} and develop the notion in the setting of reflexive Banach spaces. We also consider the concept of restriction, recently considered by I. Oppenheim in the context of $L^2$-cohomology \cite{IO}. Proposition \ref{averagenorms} and \ref{propQ} are the key results. The former relates the norm of the average to the norm of the differential, whereas the latter gives a global vanishing condition in the kernel of the full codifferential. 
\begin{definition}
For a weight $\omega$, define the localized weight as $$\omega_\tau(\sigma) = \omega(\tau * \sigma)$$ for $\sigma \in \Sigma_\tau(j)$ and $\tau \in \Sigma(l)$ such that $\tau * \sigma \in \Sigma(j+l+1)$. 
\end{definition}

In other words, for $\tau \in \Sigma(l)$, $\omega_\tau(\sigma)$ is the number of $(n-l-1)$-simplexes in $X_\tau$ containing $\sigma \in \Sigma_\tau(j)$. 


\begin{lemma}
$\Gamma_\tau$ acts by simplicial automorphisms on $X_\tau$. 
\end{lemma}
\begin{proof} Let $\sigma \in X_\tau$ and suppose $g \in \Gamma_\tau$. Since, $\sigma \subset \tau * \sigma$, the join of $\sigma$ and $\tau$, it follows that $g \cdot \sigma \subset g \cdot \tau * g \cdot \sigma = \tau * g \cdot \sigma$ as $\Gamma$ acts by simplicial automorphisms on $X$ and $g \in \Gamma_\tau$. Thus, $g \cdot \sigma$ is a simplex in $\tau * g \cdot \sigma$, and so $g \cdot \sigma \in X_\tau$ since it is disjoint from $\tau$. Hence, $\Gamma_\tau$ act by simplicial automorphisms on $X_\tau$. \end{proof}

\begin{lemma}
For $\eta \in X_\tau$, $\Gamma_{\tau \eta} = \Gamma_\tau \cap \Gamma_\eta$.
\end{lemma}
\begin{proof} $\Gamma_{\tau \eta} = \lbrace g \in \Gamma_\tau \colon g \cdot \eta = \eta \rbrace = \Gamma_\tau \cap \Gamma_\eta.$ \end{proof}

\begin{definition}
We denote by 
\begin{enumerate}
\item[i.] $\pi_\tau$ the restriction of $\pi$ to $\Gamma_\tau$, that is $\pi_\tau = \pi \vert_{\Gamma_\tau}$;
\item[ii.] $d_\tau$ the restriction of $d$ to $\mathcal{E}^{(k,p)}(X_\tau, E)$, that is $d_\tau = d \vert_{\mathcal{E}^{(k,p)}(X_\tau, E)}$;
\item[iii.] $\delta_\tau$ the restriction of $\delta$ to $L^{(k+1,p^*)}(X_\tau, E^*)$, that is $\delta_\tau = \delta \vert_{L^{(k+1,p^*)}(X_\tau, E^*)}$. 
\end{enumerate}
\end{definition}

\begin{definition}
Let $$ \mathcal{E}^{(k,p)}(X_\tau,E) = \left\lbrace f \colon \Sigma_\tau(k) \rightarrow E \colon \Vert f \Vert^p_{(k,p)} < \infty \right\rbrace$$ for $X_\tau \subset X$ denote the vector space of $p$-summable functions with semi-norm $$\Vert f \Vert_{(k,p)} = \left( \sum_{\sigma \in \Sigma_\tau(k, \Gamma_\tau)} \Vert f(\sigma) \Vert^p_E \dfrac{\omega_\tau(\sigma)}{(k+1)! \vert \Gamma_{\tau \sigma}\vert}\right)^{\dfrac{1}{p}}.$$
\end{definition}


\begin{definition}
Let $L^{(k,p)}(X_\tau, E)$ denote the subspace $$\left\lbrace f \in \mathcal{E}^{(k,p)}(X_\tau,E)\, \colon \, f \, \mathrm{alternating} \, \mathrm{and} \, \forall g \in \Gamma_\tau, \forall \sigma \in \Sigma_\tau(k), f(g \cdot \sigma) = \pi_{\tau g} \cdot f(\sigma) \right\rbrace$$ of simplicial $k$-cochains of $X_\tau$ twisted by $\pi_\tau$.  
\end{definition}

\begin{definition}
For $f \in \mathcal{E}^{(k,p)}(X,E)$ and $\tau \in \Sigma(j)$ such that $k-j-1 \geq 0$, the localization of $f$ to $X_\tau$ is the function $f_\tau \in \mathcal{E}^{(k,p)}(X_\tau,E) \in \mathcal{E}^{(k-j-1, p)}(X_\tau, E)$ defined by the localization map $${}_\tau \colon \mathcal{E}^{(k,p)}(X,E) \longrightarrow \mathcal{E}^{(k-j-1,p)}(X_\tau,E)$$ where for all $\sigma \in \Sigma_\tau(k-j-1)$, $f_\tau(\sigma) = f (\tau * \sigma)$. Similarly we define its dual $\mathcal{E}^{(k,p^*)}(X,E^*) \longrightarrow \mathcal{E}^{(k,p^*)}(X_\tau,E^*)$, also denoted by ${}_\tau$.
\end{definition}

\begin{definition}
For $f \in \mathcal{E}^{(k,p)}(X,E)$ and $\tau \in \Sigma(j)$ such that $k+j+1 \leq n$, the restriction of $f$ to $X_\tau$ is the function $f^\tau \in \mathcal{E}^{(k,p)}(X_\tau,E)$ defined by the restriction map $${}^\tau \colon \mathcal{E}^{(k,p)}(X,E) \rightarrow \mathcal{E}^{(k,p)}(X_\tau,E)$$ where for all $\sigma \in \Sigma_\tau(k)$, $f^\tau (\sigma) = f(\sigma)$. Similarly we define its dual $\mathcal{E}^{(k,p^*)}(X,E^*) \longrightarrow \mathcal{E}^{(k-j-1,p^*)}(X_\tau,E^*)$, also denoted by ${}^\tau$.
\end{definition}

Next, we consider a number of local to global equalities that will be of use. We begin by the following useful local relation: 
\begin{proposition} \label{BS1.10}
For $f_\tau \in L^{(k,p)}(X_\tau, E)$ $$\Vert f_\tau \Vert^p_{(k,p)} = \dfrac{1}{(k+1)!\vert \Gamma_\tau \vert} \sum_{\sigma \in \Sigma_\tau(k)} \Vert f_\tau(\sigma) \Vert^p_E \omega(\tau * \sigma).$$
\end{proposition}
Proof. \begin{align}
\Vert f_\tau \Vert^p_{(k,p)} &= \sum_{\sigma \in \Sigma_{\tau}(k, \Gamma_\tau)} \Vert f_\tau(\sigma) \Vert^p_E \dfrac{\omega_\tau(\sigma)}{\vert \Gamma_{\tau \sigma}\vert(k+1)!} \substack{{} \\ =} \sum_{\sigma \in \Sigma_{\tau}(k)} \dfrac{\Vert f_\tau(\sigma) \Vert^p_E}{\vert \Gamma_\tau \sigma \vert} \dfrac{\omega_\tau(\sigma)}{\vert \Gamma_{\tau \sigma}\vert(k+1)!} \nonumber \\ &\substack{{} \\ =} \sum_{\sigma \in \Sigma_{\tau}(k)} \Vert f_\tau(\sigma) \Vert^p_E \dfrac{\vert \Gamma_{\tau\sigma}\vert}{\vert\Gamma_\tau\vert} \dfrac{\omega_\tau(\sigma)}{\vert \Gamma_{\tau \sigma}\vert(k+1)!} = \sum_{\sigma \in \Sigma_{\tau}(k)} \Vert f_\tau(\sigma) \Vert^p_E  \dfrac{\omega_\tau(\sigma)}{\vert \Gamma_{\tau}\vert(k+1)!}, \nonumber \\ &= \dfrac{1}{(k+1)! \vert \Gamma_\tau \vert}\sum_{\sigma \in \Sigma_{\tau}(k)} \Vert f_\tau(\sigma) \Vert^p_E {\omega}(\tau *\sigma) \nonumber
\end{align} where the second equality in terms of the $k$-simplexes of $X_\tau$ follows from the $\Gamma_\tau$ invariance of the norm, and the third by the fact that $\vert \Gamma_\tau \sigma \vert$, the size of the $\Gamma_\tau$ orbit of $\sigma$, is $\vert \Gamma_\tau \vert / \vert \Gamma_{\tau \sigma} \vert$. \begin{flushright}$\Box$ \end{flushright}


\begin{proposition} \label{oppenheim}
Let $f \in L^{(k,p)}(X,E)$. If $k + 1 \leq n$, then 
$$(n-k) \Vert f \Vert_{(k,p)}^p = \sum_{\tau \in \Sigma(0, \Gamma)} \Vert f^\tau \Vert_{(k,p)}^p.$$
\end{proposition}
\begin{proof} For $\xi \in \Sigma(k+1)$ such that $\tau \subset \xi$, denote by $\xi-\tau$ the $k$-simplex in $X_\tau$ obtained by removing the vertex $\tau$ from $\xi$. Now,
\begin{align}
\sum_{\tau \in \Sigma(0, \Gamma)} \Vert f^\tau \Vert_{(k,p)}^p &\substack{{} \\ =} \sum_{\tau \in \Sigma(0, \Gamma)} \dfrac{1}{(k+1)! \vert \Gamma_\tau \vert} \sum_{\sigma \in \Sigma_\tau(k)} \Vert f (\sigma) \Vert_E^p \omega (\tau * \sigma) \nonumber \\ &\substack{{} \\ =} \sum_{\tau \in \Sigma(0, \Gamma)} \dfrac{1}{(k+1)! \vert \Gamma_\tau \vert} \sum_{\substack{\xi \in \Sigma(k+1) \\ \tau \subset \xi}} \dfrac{1}{(k+2)} \Vert f (\xi - \tau) \Vert_E^p \omega (\xi) \nonumber \\ &= \sum_{\tau \in \Sigma(0, \Gamma)} \sum_{\substack{\xi \in \Sigma(k+1) \\ \tau \subset \xi}} \dfrac{1}{(k+2)! \vert \Gamma_\tau \vert}  \Vert f (\xi - \tau) \Vert_E^p \omega (\xi) \nonumber \\ &\substack{(*) \\ =} \sum_{\xi \in \Sigma(k+1, \Gamma)} \sum_{\substack{\tau \in \Sigma(0) \\ \tau \subset \xi}} \dfrac{1}{(k+2)! \vert \Gamma_\xi \vert}  \Vert f (\xi - \tau) \Vert_E^p \omega (\xi) \nonumber \\ &\substack{{} \\ =} \sum_{\xi \in \Sigma(k+1, \Gamma)} \sum_{\substack{\sigma \in \Sigma(k) \\ \sigma \subset \xi}} \dfrac{1}{(k+2)! \vert \Gamma_\xi \vert} \dfrac{1}{(k+1)!} \Vert f (\sigma) \Vert_E^p \omega (\xi) \nonumber \\ &\substack{(**) \\ =} \sum_{\sigma \in \Sigma(k, \Gamma)} \sum_{\substack{\xi \in \Sigma(k+1) \\ \sigma \subset \xi}} \dfrac{1}{(k+2)! \vert \Gamma_\sigma \vert} \dfrac{1}{(k+1)!} \Vert f (\sigma) \Vert_E^p \omega (\xi) \nonumber \\ &\substack{{} \\ =} \sum_{\sigma \in \Sigma(k, \Gamma)} \dfrac{1}{(k+2)! \vert \Gamma_\sigma \vert} \dfrac{1}{(k+1)!} \Vert f (\sigma) \Vert_E^p (n-k)(k+2)!\omega (\sigma) \nonumber \\ &= (n-k) \Vert f \Vert_{(k,p)}^p, \nonumber
\end{align}
where the first and second equality follow by Proposition \ref{BS1.10}. noting that $f^\tau (\sigma) = f(\sigma)$, writing $\xi - \tau$ as $\sigma$ and accounting for ordering. $(*)$ and $(**)$ follows by switching sums by Proposition \ref{switchingsums} and the second last equality follows by Proposition \ref{combinatorial}. \end{proof} 


\begin{proposition} \label{summing}
Let $f \in L^{(k,p)}(X, E)$ and $0 \leq j < k$. Then, $$(k+1)! \Vert f \Vert^p_{(k,p)} = (k-j)! \sum_{\tau \in \Sigma(j, \Gamma)} \Vert f_\tau \Vert^p_{(k-j-1,p)}.$$
\end{proposition}
\begin{proof} 
\begin{align}
\sum_{\tau \in \Sigma(j, \Gamma)} \Vert f_\tau \Vert^p_{(k-j-1,p)} &\substack{{} \\ =} \sum_{\tau \in \Sigma(j, \Gamma)} \sum_{\eta \in \Sigma_\tau(k-j-1)} \dfrac{\omega(\tau *\eta)}{(k-j)! \vert \Gamma_\tau\vert} \Vert f_\tau(\eta)\Vert^p_E \nonumber \\ &\substack{{} \\ =} \sum_{\tau \in \Sigma(j, \Gamma)} \sum_{\substack{\sigma \in \Sigma(k) \\ \sigma = \tau * \eta}} \dfrac{\omega(\sigma)}{(k-j)!\vert \Gamma_\tau \vert} \Vert f(\sigma) \Vert^p_E \nonumber \\ &\substack{(*) \\ =} \sum_{\tau \in \Sigma(j, \Gamma)} \sum_{\substack{\sigma \in \Sigma(k) \\ \tau \subset \sigma}} \dfrac{(k-j)!}{(k+1)!} \dfrac{\omega(\sigma)}{(k-j)!\vert \Gamma_\tau \vert} \Vert f(\sigma) \Vert^p_E \nonumber \\ &= \sum_{\tau \in \Sigma(j, \Gamma)} \sum_{\substack{\sigma \in \Sigma(k) \\ \tau \subset \sigma}} \dfrac{\omega(\sigma)}{(k+1)!\vert \Gamma_\tau \vert} \Vert f(\sigma) \Vert^p_E \nonumber \\ &\substack{(**) \\ =} \sum_{\sigma \in \Sigma(k, \Gamma)} \sum_{\substack{\tau \in \Sigma(j) \\ \tau \subset \sigma}} \dfrac{\omega(\sigma)}{(k+1)!\vert \Gamma_\sigma \vert} \Vert f(\sigma) \Vert^p_E \nonumber \\ &\substack{(***) \\ =} \sum_{\sigma \in \Sigma(k, \Gamma)} \dfrac{(k+1)!}{(k-j)!} \dfrac{\omega(\sigma)}{(k+1)!\vert \Gamma_\sigma \vert} \Vert f(\sigma) \Vert^p_E \nonumber \\ &= \dfrac{(k+1)!}{(k-j)!} \Vert f \Vert_{(k,p)}^p, \nonumber
\end{align}
where the first and second equality follow by Proposition \ref{BS1.10} and writing $\tau * \eta$ as $\sigma \in \Sigma(k)$, respectively. On the other hand, $(*)$ follows since summing over all $\sigma \in \Sigma(k)$ such that $\tau \subset \sigma$ amounts to summing over each term in the previous sum $(k+1)! / ((k+1)-(j+1))! = (k+1)!/(k-j)!$ times recalling that $\omega$ is symmetric and $f$ alternating. $(**)$ follows by Proposition \ref{switchingsums}, and finally $(***)$ follows since there are $(k+1)!/(k-j)!$ terms independent of $\tau$ in the sum over all $\tau \in \Sigma(j)$ with vertices in $\sigma$.
\end{proof}

\begin{corollary} \label{combinedsumming}
Suppose $f \in L^{(k,p)}(X,E)$. If $1 < k + 1 \leq n$, then $$ \sum_{\tau \in \Sigma(0, \Gamma)} \Vert f^\tau \Vert_{(k,p)}^p = \dfrac{n-k}{k+1} \sum_{\tau \in \Sigma(0, \Gamma)} \Vert f_\tau \Vert_{(k-1,p)}^p.$$
\end{corollary}
\begin{proof} Follows immediately by Proposition \ref{oppenheim} and Proposition \ref{summing} above in the case $j=0$, 
\begin{align}
\dfrac{k!}{(k+1)!} \sum_{\tau \in \Sigma(0, \Gamma)} \Vert f_\tau \Vert_{(k-1,p)}^p = \Vert f \Vert^p_{(k,p)}  = \dfrac{1}{(n-k)}  \sum_{\tau \in \Sigma(0, \Gamma)} \Vert f^\tau \Vert^p_{(k,p)}. \nonumber
\end{align} \end{proof}

\begin{definition}
Define the localized average over a cochain $\phi$ by the map $$M : L^{(k,p)}(X_\tau,E) \rightarrow L^{(k,p)}(X_\tau,E)$$ $\phi_\tau \mapsto M \phi_\tau = \phi_\tau^0$ where $$\phi_\tau^0 (\sigma) = \left( \sum_{\sigma \in \Sigma_\tau(k)} \omega_\tau(\sigma) \right)^{-1} \sum_{\sigma \in \Sigma_\tau(k)} \omega_\tau(\sigma) \phi_\tau(\sigma).$$ Similarly, we define its dual as $\overline{M} : L^{(k,p^*)}(X_\tau,E^*) \rightarrow L^{(k,p^*)}(X_\tau,E^*)$. 
\end{definition}

\begin{corollary} \label{Mbound}
The map $M : L^{(k,p)}(X_\tau,E) \rightarrow L^{(k,p)}(X_\tau,E)$ and its dual $\overline{M}$ are bounded projections onto the space of constant maps. 
\end{corollary}
\begin{proof} $M$ is well-defined. Towards this end, let $\phi_\tau \in L^{(k,p)}(X_\tau,E)$. Since $\omega_\tau$ is symmetric and $\Gamma_\tau$-invariant, and $\phi_\tau$ is alternating and twisted by $\pi_\tau$, $M \phi_\tau$ is alternating and twisted by $\pi_\tau$ as a finite weighted sum of such functions. Moreover,
\begin{align}
\Vert M \phi_\tau \Vert^p_{(k,p)} &= \dfrac{1}{(k+1)! \vert \Gamma_\tau \vert} \sum_{\eta \in \Sigma_\tau(k)} \omega_\tau(\eta) \Vert \phi^0_\tau (\eta) \Vert^p_E \nonumber \\ &= \dfrac{1}{(k+1)! \vert \Gamma_\tau \vert} \left( \sum_{\sigma \in \Sigma_\tau(k)} \omega_\tau(\sigma) \right)^{-p} \left\Vert \sum_{\sigma \in \Sigma_\tau(k)} \omega_\tau(\sigma) \phi_\tau(\sigma) \right\Vert_E^p \sum_{\eta \in \Sigma_\tau(k)} \omega_\tau(\eta) \nonumber \\ &= \dfrac{1}{(k+1)! \vert \Gamma_\tau \vert} \left( \sum_{\sigma \in \Sigma_\tau(k)} \omega_\tau(\sigma) \right)^{1-p} \left\Vert \sum_{\sigma \in \Sigma_\tau(k)} \omega_\tau(\sigma) \phi_\tau(\sigma) \right\Vert_E^p \nonumber \\ &\leq \dfrac{C^p}{(k+1)! \vert \Gamma_\tau \vert} \left( \sum_{\sigma \in \Sigma_\tau(k)} \omega_\tau(\sigma) \right)^{1-p} \left\Vert \sum_{\sigma \in \Sigma_\tau(k)} \phi_\tau(\sigma) \right\Vert_E^p \nonumber \\ &\leq \dfrac{C^p \vert \Sigma_\tau(k) \vert^p}{(k+1)! \vert \Gamma_\tau \vert} \left( \sum_{\sigma \in \Sigma_\tau(k)} \omega_\tau(\sigma) \right)^{1-p} \sum_{\sigma \in \Sigma_\tau(k)} \left\Vert \phi_\tau(\sigma) \right\Vert_E^p \nonumber \\ &= \dfrac{C^p \vert \Sigma_\tau(k) \vert^p}{(k+1)! \vert \Gamma_\tau \vert} \left( \sum_{\sigma \in \Sigma_\tau(k)} \omega_\tau(\sigma) \right)^{1-p} \sum_{\sigma \in \Sigma_\tau(k)} \dfrac{\omega_\tau(\sigma)}{\omega_\tau(\sigma)}\left\Vert \phi_\tau(\sigma) \right\Vert_E^p \nonumber \\ &\leq \dfrac{C^p \vert \Sigma_\tau(k) \vert^p}{D (k+1)! \vert \Gamma_\tau \vert} \left( \sum_{\sigma \in \Sigma_\tau(k)} \omega_\tau(\sigma) \right)^{1-p} \sum_{\sigma \in \Sigma_\tau(k)} \omega_\tau(\sigma)\left\Vert \phi_\tau(\sigma) \right\Vert_E^p \nonumber \\ &\leq \dfrac{C^p \vert \Sigma_\tau(k) \vert^p}{D} \vert \Sigma_\tau(k) \vert^{1-p} D^{1-p} \Vert \phi_\tau \Vert^p_{(k,p)} \nonumber \\ &= \dfrac{(C/D)^p}{\vert \Sigma_\tau(k) \vert} \Vert \phi_\tau \Vert^p_{(k,p)} \leq (C/D)^p \Vert \phi_\tau \Vert^p_{(k,p)} \nonumber
\end{align}
where
\begin{enumerate}
\item[] $C = \max \lbrace \omega_\tau(\sigma) \colon \sigma \in \Sigma_\tau(k) \rbrace$ which exists as $\Sigma_\tau (k)$ contains only finitely many $k$-simplexes;
\item[] $D = \min \lbrace \omega_\tau(\sigma) \colon \sigma \in \Sigma_\tau(k) \rbrace$.
\end{enumerate}
Hence, $M \phi_\tau \in L^{(k,p)}(X_\tau,E)$ and $M$ is well-defined and bounded. Clearly $M$ is linear and 
\begin{align} M^2 \phi_\tau &= M \phi_\tau^0 \nonumber \\ &= \left( \sum_{\sigma \in \Sigma_\tau(k)} \omega_\tau(\sigma) \right)^{-1} \sum_{\sigma \in \Sigma_\tau(k)} \omega_\tau(\sigma) \phi_\tau^0(\sigma) \nonumber \\ &= \left( \sum_{\sigma \in \Sigma_\tau(k)} \omega_\tau(\sigma) \right)^{-1} \sum_{\sigma \in \Sigma_\tau(k)} \omega_\tau(\sigma) \left( \sum_{\eta \in \Sigma_\tau(k)} \omega_\tau(\eta) \right)^{-1} \sum_{\eta \in \Sigma_\tau(k)} \omega_\tau(\eta) \phi_\tau(\eta) \nonumber \\ &= \left( \sum_{\eta \in \Sigma_\tau(k)} \omega_\tau(\eta) \right)^{-1} \sum_{\eta \in \Sigma_\tau(k)} \omega_\tau(\eta) \phi_\tau(\eta) \nonumber \\ &= M \phi_\tau, \nonumber
\end{align} so $M$ is a continuous projection onto $ \lbrace f : \Sigma_\tau(k) \rightarrow E  \colon f = \, \mathrm{constant} \rbrace \subseteq L^{(k,p)}(X_\tau,E)$. Similarly for $\overline{M}$ \end{proof}


\begin{proposition} \label{averagenorms}
Let $0 \leq j < k \leq n$, $\tau \in \Sigma(j)$ and $\phi \in L^{(k,p^*)}(X,E^*)$. Then,
\begin{enumerate}
\item if $j < k-1$, then $\delta_\tau \phi_\tau = (-1)^{j+1}(\delta \phi)_\tau$;
\item if $j= k-1$, then $(-1)^k(n-k+1) \phi^0_\tau = \delta \phi(\tau)$ and
$$ \Vert \phi^0_\tau \Vert^{p^*}_{(0,p^*)} = \dfrac{\omega(\tau)}{(n-k+1)^{p^*-1}\vert \Gamma_\tau \vert} \Vert \delta \phi(\tau) \Vert^{p^*}_{E^*}.$$
\end{enumerate}
\end{proposition}
\begin{proof} $(1)$ As $\phi \in L^{(k,p^*)}(X,E^*)$, it follows that $\phi_\tau \in L^{(k-j-1,p^*)}(X_\tau,E^*)$ where $k-j-1 > 0$ so $\delta_\tau \phi_\tau \in L^{(k-j-2,p^*)}(X_\tau, E^*)$ and by Proposition \ref{codifferential}, 
\begin{align}
\delta_\tau \phi_\tau (\sigma) &\substack{(*) \\ =} \sum_{\substack{v \in \Sigma_\tau(0) \\ v * \sigma \in \Sigma_\tau(k-j-1)}} \dfrac{\omega_\tau(v * \sigma)}{\omega_\tau(\sigma)} \phi_\tau(v * \sigma) \nonumber \\ &= \sum_{\substack{v \in \Sigma_\tau(0) \\ v * \sigma \in \Sigma_\tau(k-j-1)}} \dfrac{{\omega}(\tau * v * \sigma)}{{\omega}(\tau * \sigma)} \phi(\tau * v * \sigma) \nonumber \\ &\substack{(**) \\ =} \sum_{\substack{v \in \Sigma_\tau(0) \\ v * \sigma \in \Sigma_\tau(k-j-1)}} (-1)^{j+1} \dfrac{{\omega}(v * \tau * \sigma)}{{\omega}(\tau * \sigma)} \phi(v * \tau * \sigma) \nonumber \\ &= \sum_{\substack{v \in \Sigma(0) \\ v * \tau * \sigma \in \Sigma(k)}} (-1)^{j+1} \dfrac{{\omega}(v * \tau * \sigma)}{{\omega}(\tau * \sigma)} \phi(v * \tau * \sigma) \nonumber \\ &= (-1)^{j+1} \delta \phi(\tau * \sigma) = (-1)^{j+1} \left( \delta \phi \right)_\tau (\sigma) \nonumber
\end{align}
where $(*)$ follows by Proposition \ref{codifferential} and $(**)$ holds since $\omega$ is symmetric, $\phi$ alternating and $\tau \in \Sigma(j)$.
As for $(2)$, by Proposition \ref{codifferential} together with the fact that $\omega$ is symmetric and $\phi$ antisymmetric it follows for the codifferential that
\begin{align}
\delta\phi(\tau) &= \dfrac{1}{{\omega}(\tau)} \sum_{\sigma \in \Sigma_\tau(0)} \omega(\sigma * \tau) \phi (\sigma * \tau) = \dfrac{1}{{\omega}(\tau)} \sum_{\sigma \in \Sigma_\tau(0)} (-1)^{j+1}\omega(\tau * \sigma) \phi (\tau * \sigma) \nonumber \\ &= \dfrac{1}{{\omega}(\tau)} \sum_{\sigma \in \Sigma_\tau(0)} (-1)^{k}\omega(\tau * \sigma) \phi (\tau * \sigma), \nonumber
\end{align} 
since $j=k-1$. Therefore, in terms of $\phi_\tau^0 \in L^{(0,p^*)}(X_\tau,E^*)$
\begin{align}
\delta \phi(\tau) = \dfrac{(-1)^k}{{\omega}(\tau)} \sum_{\sigma \in \Sigma_\tau(0)} \omega_\tau(\sigma) \phi_\tau (\sigma) = \dfrac{(-1)^k}{{\omega}(\tau)} \left( \sum_{\sigma \in \Sigma_\tau(0)} \omega_\tau(\sigma)\right) \phi_\tau^0. \nonumber
\end{align}
However, 
\begin{align}
\dfrac{\displaystyle \sum_{\sigma \in \Sigma_\tau(0)} \omega_\tau(\sigma)}{\omega(\tau)} = \dfrac{\displaystyle \sum_{\tau * \sigma \in \Sigma(j+1)} {\omega}(\tau * \sigma)} {\omega(\tau)} = \dfrac{(n-j)(j+1)!}{(j+1)!} = (n-j) = (n-k+1) \nonumber
\end{align}
by Proposition \ref{combinatorial}. The factor $(j+1)!$ in the denominator corresponds to the fact that we sum over one ordering as $\tau$ is fixed. Therefore, $$\delta \phi(\tau) = (-1)^k(n-k+1)\phi_\tau^0,$$ and once again by Proposition \ref{BS1.10} this gives
\begin{align}
\Vert \phi^0_\tau \Vert^{p^*}_{(0,p^*)} &= \dfrac{1}{\vert \Gamma_\tau \vert} \sum_{\sigma \in \Sigma_\tau(0)} \Vert \phi^0_\tau(\sigma)\Vert^{p^*}_{E^*} \omega_\tau(\sigma) = \dfrac{1}{\vert \Gamma_\tau \vert} \sum_{\sigma \in \Sigma_\tau(0)} \dfrac{\Vert \delta \phi(\tau) \Vert^{p^*}_{E^*}}{(n-k+1)^{p^*}} \omega_\tau(\sigma) \nonumber \\ &=  \dfrac{1}{\vert \Gamma_\tau \vert} \dfrac{\Vert \delta \phi(\tau) \Vert^{p^*}_{E^*}}{(n-k+1)^{p^*}} (n-k+1) \omega(\tau) = \dfrac{\omega(\tau) \Vert \delta \phi(\tau) \Vert^{p^*}_{E^*}}{(n-k+1)^{p^*-1}\vert \Gamma_\tau \vert}. \nonumber
\end{align} \end{proof}

\begin{proposition} \label{try1}
Let $\phi \in L^{(k,p)}(X,E)$ and $\tau \in \Sigma(k-1)$. Then,
\begin{enumerate}
\item if $k=1$, $$ d_\tau \phi_\tau (\sigma) = -(d \phi)_\tau (\sigma) + \phi (\sigma).$$
\item if $k > 1$, $$ d_\tau \phi_\tau (\sigma) =  (d \phi)_\tau (\sigma) + \sum_{i=0}^{k-1}(-1)^i\phi (\tau_i * \sigma).$$
\end{enumerate} 
\end{proposition}
\begin{proof} Suppose $k=1$. Then, for $\sigma = x * y \in \Sigma_\tau(1)$ 
\begin{align}
d_\tau \phi_\tau (\sigma) &= \phi_\tau (y) -\phi_\tau (x) = \phi(\tau * y) - \phi(\tau * x). \nonumber
\end{align} 
On the other hand, as $[\tau * x * y : x * y] =1$
\begin{align}
d \phi (\tau * x * y) &= [\tau * x * y : x * y] \phi(x*y) + [\tau * x * y: \tau * y] \phi (\tau * y) + [\tau * x * y: \tau * x] \phi(\tau * x)  \nonumber \\ &= \phi(x*y) - \phi (\tau * y) +\phi(\tau * x), \nonumber
\end{align}
gives together with the expression for $d_\tau \phi_\tau$ 
$$ d_\tau \phi_\tau (\sigma) = -d \phi (\tau * x * y) + \phi (x * y) = -(d \phi)_\tau (\sigma) + \phi (\sigma).$$ 
Suppose $k>1$ and $\tau \in \Sigma(k-1)$. Then, as previously
\begin{align}
d_\tau \phi_\tau (\sigma) = \phi_\tau (y) - \phi_\tau (x), \nonumber
\end{align}
and the two rightmost terms are as previously the last two terms in $d \phi (\tau * \sigma)$. 
\end{proof}

\begin{proposition} \label{diffrelnew}
Let $\phi \in L^{(k,p)}(X,E)$ and $\tau \in \Sigma(0)$, then $$(d \phi)_\tau (\sigma) = - d_\tau \phi_\tau (\sigma) + \phi(\sigma).$$
\end{proposition}
\begin{proof} Let $\sigma \in \Sigma(k)$. Then, similarly as in Proposition \ref{try1}
$$d_\tau \phi_\tau (\sigma) = \sum_{i=0}^{k}(-1)^i \phi_\tau (\sigma_i) = \sum_{i=0}^{k}(-1)^i \phi(\tau * \sigma_i) = - (d \phi)_\tau(\sigma) + \phi(\sigma).$$ \end{proof}

\begin{corollary} \label{help1}
Suppose $\phi \in L^{(k,p)}(X,E)$, $\tau \in \Sigma(0)$ and $k+1 \leq n$. Then, if $\phi \in \ker d$, $$ \Vert d_\tau \phi_\tau \Vert_{(k,p)} = \Vert \phi^\tau \Vert_{(k,p)}.$$
\end{corollary}
\begin{proof} By Proposition \ref{diffrelnew} it follows since $\phi \in \ker d$ that
\begin{align}
\Vert d_\tau \phi_\tau \Vert_{(k,p)}^p &= \sum_{\sigma \in \Sigma_\tau (k, \Gamma_\tau)} \Vert d_\tau \phi_\tau (\sigma)\Vert_E^p \dfrac{\omega_\tau(\sigma)}{(k+1)! \vert \Gamma_{\tau \sigma}\vert} \nonumber \\ &= \sum_{\sigma \in \Sigma_\tau (k, \Gamma_\tau)} \Vert \phi (\sigma)\Vert_E^p \dfrac{\omega_\tau(\sigma)}{(k+1)! \vert \Gamma_{\tau \sigma}\vert} \nonumber \\ &= \Vert \phi^\tau \Vert_{(k,p)}^p. \nonumber
\end{align} \end{proof}

\begin{corollary} \label{try2}
Suppose $\phi \in L^{(k,p)}(X,E)$ and $1 < k+1 \leq n$. If $\phi \in \ker d$, then $$ \sum_{\tau \in \Sigma(0, \Gamma)} \Vert d_\tau \phi_\tau \Vert_{(k,p)}^p = \dfrac{n-k}{k+1} \sum_{\tau \in \Sigma(0, \Gamma)} \Vert \phi_\tau \Vert_{(k-1,p)}^p.$$
\end{corollary}
\begin{proof} Follows by Corollary \ref{help1} and Proposition \ref{combinedsumming}. \end{proof}

\begin{corollary} \label{BSspecial}
Suppose $\phi \in L^{(1,p)}(X,E)$. If $\phi \in \ker d$, then
$$- (n-1) \Vert \phi \Vert_{(1,p)}^p = \sum_{\tau \in \Sigma(0, \Gamma)} \left( \Vert d_\tau \phi_\tau \Vert_{(1,p)}^p - (n-1) \Vert \phi_\tau \Vert^p_{(0,p)} \right).$$
\end{corollary}
\begin{proof} By a direct computation,
\begin{align}
&\sum_{\tau \in \Sigma(0, \Gamma)} \Vert d_\tau \phi_\tau \Vert_{(1,p)}^p - (n-1) \Vert \phi_\tau \Vert_{(0,p)}^p \substack{(*) \\ =} \sum_{\tau \in \Sigma(0, \Gamma)} \dfrac{(n-1)}{2} \Vert \phi_\tau\Vert_{(0,p)}^p - (n-1) \Vert \phi_\tau\Vert_{(0,p)}^p \nonumber \\ &= -\dfrac{(n-1)}{2} \sum_{\tau \in \Sigma(0, \Gamma)} \Vert \phi_\tau \Vert_{(0,p)}^p \substack{(**) \\ =} - \dfrac{(n-1)}{2} \dfrac{2!}{(1-0)!} \Vert \phi \Vert_{(1,p)}^p = -(n-1)\Vert \phi \Vert_{(1,p)}^p, \nonumber
\end{align}
where in $(*)$ we used Corollary \ref{try2} and in $(**)$ Proposition \ref{summing}.
\end{proof}

\begin{corollary} \label{try3}
Suppose $\phi \in L^{(1,p)}(X,E)$. If $\phi \in \ker d$, then
$$\sum_{\tau \in \Sigma(0, \Gamma)} \Vert d_\tau \phi_\tau \Vert_{(1,p)}^p = (n-1) \Vert \phi \Vert_{(1,p)}^p.$$
\end{corollary}
\begin{proof} Follows directly from Corollary \ref{BSspecial} using Proposition \ref{summing} once more. \end{proof}

\begin{definition} \label{Q}
Let $\phi \in L^{(1,p)}(X,E)$. For $\tau \in \Sigma(0)$ define a $p$-form on $L^{(0,p)}(X_\tau, E)$ by
$$Q_\tau (\phi_\tau) = \Vert d_\tau \phi_\tau \Vert^p_{(1,p)} - \dfrac{(n-1)}{2} \Vert \phi_\tau \Vert_{(0,p)}^p.$$ Similarly, we define a $p^*$-form on $L^{(0,p^*)}(X_\tau, E^*)$.
\end{definition}

\begin{corollary} \label{propQ}
Suppose $\phi \in L^{(1,p)}(X,E)$. If $\phi \in \ker d$ then
$$ \sum_{\tau \in \Sigma(0, \Gamma)} Q_\tau (\phi_\tau) = 0.$$
\end{corollary}
\begin{proof} Follows immediately by Corollary \ref{try3} and Proposition \ref{summing}. \end{proof}


\section{Poincar\'e inequalities on finite weighted graphs} \label{PI}
In this section we recall some basic facts concerning Poincar\'e inequalities on finite weighted graphs necessary for the spectral method. For details we refer to \cite{gn, Nowak, ny}.
\begin{proposition} \cite{Nowak} \label{poink}
Suppose $\dim X = 2$. Then the link $X_\tau$ of every vertex of $X$ is a finite graph. Hence, for any $p \geq 1$ the $p$-Poincar\'e inequality \begin{align} \sum_{\sigma \in \Sigma(0, \Gamma_\tau)} \Vert f_\tau(\sigma) - Mf_\tau(\sigma)\Vert_E^p &\dfrac{{\omega_\tau}(\sigma)}{\vert \Gamma_{\tau \sigma} \vert} \nonumber \\ &\leq \kappa_p^p \sum_{\eta \in \Sigma_\tau(1, \Gamma_\tau)} \dfrac{1}{2} \Vert f_\tau(\eta_0) - f_\tau(\eta_1) \Vert_E^p \dfrac{\omega_\tau(\eta)}{\vert \Gamma_{\tau \eta} \vert}  \nonumber \end{align} holds for some $\kappa_p>0$ and all $f \colon \Sigma_\tau(0) \rightarrow E$. Similarly for $f \colon \Sigma_\tau(0) \rightarrow E^*$. 
\end{proposition} \qed

The infimum of the above constants $\kappa_p$ is known as the Poincar\'e constant of the link $X_\tau$, and denoted by $\kappa_p(X_\tau, E)$. In terms of the norms introduced previously:

\begin{corollary}
Let $X$ be two dimensional. Then, for all $f \in L^{(1,p)}(X,E)$ it holds that $$\Vert f_\tau - Mf_\tau \Vert_{(0,p)} \leq \kappa_{p}(X_\tau,E) \Vert {d}_\tau f_\tau \Vert_{(1,p)},$$ for some $\kappa_{p}(X_\tau,E)$. Similarly for $f \in L^{(1,p^*)}(X,E^*)$, $$\Vert f_\tau - \overline{M}f_\tau \Vert_{(0,p^*)} \leq \kappa_{p^*}(X_\tau,E^*) \Vert \overline{d}_\tau f_\tau \Vert_{(1,p^*)},$$ for some $\kappa_{p^*}(X_\tau,E^*)$.
\end{corollary} \qed

Sometimes it is useful to know how Poincar\'e constants change under isomorphisms. The following is immediate:
\begin{proposition} \label{BI}
Let $T: E \rightarrow F$ be a Banach space isomorphism. If $$ \Vert x \Vert_E \leq \Vert T(x) \Vert_F \leq C \Vert x \Vert_E,$$ then $\kappa_p(X_\tau, E) \leq C \kappa_p(X_\tau, F)$.
\end{proposition} \qed

For $1 < p< \infty$ we denote by ${L}^p$ the Banach space ${L}^p(\mu)$ of $p$-integrable functions on a standard Borel space $(Y, \mathcal{B})$ with $\sigma$-finite measure $\mu$. As such, any separable infinite-dimensional Hilbert space $\mathcal{H}$ is isometrically isomorphic to ${L}^2$. In particular, we have the following relation between the Poincar\'e constant and spectral gap:

\begin{proposition} \label{hilb1}
\cite{Nowak} Let $\lambda_1$ be the smallest positive eigenvalue of the graph Laplacian $\triangle_+ = (\delta d)_\tau$, defined by $$\triangle_+f(v) = f(v) - \dfrac{1}{\omega_\tau(v)} \sum_{u \in L_v}f(u),$$ where $L_v$ denotes the link of $v$ in $X_\tau$, over the space $C^{(0,p)}(X_\tau, \mathbb{R})$ of real-valued functions on the vertices. For $L^2$ when $X$ is $2$-dimensional, $\kappa_p(X_\tau, L^2) = \lambda_1^{-1/2}$ and more generally $\kappa_p(X_\tau, L^p) = \kappa_p(X_\tau, \mathbb{R})$.
\qed \end{proposition}

\section{$L^p\,$-cohomology and vanishing for uniformly bounded representations} \label{LH}
Following \cite{BS} we introduce $L^p$-cohomology of $X$ with coefficients in $\pi$ as a natural extension of the $L^2$-cohomology for unitary representations. In particular, if $\pi$ is a unitary representation $L^2 H^k(X, \pi)$, as described below, is the cohomology of the complex of $\textrm{mod}$ $\Gamma$ square integrable cochains of $X$ twisted by $\pi$. The connection to property $(T)$ is as follows: if $X$ is a two dimensional contractible simplicial complex and $\Gamma$ acts properly discontinuously and cocompactly by automorphisms on it, then $\Gamma$ has property $(T)$ if and only if $L^2 H^1(X, \pi) = 0$ for any unitary representation \cite{hv}. As an application we derive a spectral condition for cohomological vanishing for square integrable cochains on a two dimensional simplicial complex twisted by a uniformly bounded representation. 
\begin{definition}
Let $$L^p H^k(X,\pi) = \ker \left( d \vert_{L^{(k,p)}(X,E)} \right) / \mathrm{im} \left( d \vert_{L^{(k-1,p)}(X,E)} \right)$$ denote the $L^p$-cohomology groups of $X$ with coefficients twisted by $\pi$. 
\end{definition} 
As the following shows, cohomological vanishing takes place when $\delta$ is bounded from below:
\begin{proposition} \label{VCH1}
The map $$d_{k-1} \vert_{L^{(k-1,p)}(X,E)} : L^{(k-1,p)}(X,E) \longrightarrow \ker \, d_k \vert_{L^{(k,p)}(X,E)}$$ is onto if its adjoint $$\delta_{k} : (\ker \, d_{k} \vert_{L^{(k,p)}(X,E)})^* \rightarrow L^{(k-1,p^*)}(X,E^*),$$ is bounded from below, that is $\exists \, K > 0$ such that for all $f \in \left( \ker {d_k} \vert_{L^{(k,p)}(X,E)} \right)^*$
$$\Vert \delta f \Vert_{(k-1,p^*)} \ge K \Vert f \Vert_{(k,p^*)}.$$ If in addition $d_{k-1}$ is injective, $d_{k-1}$ is onto if and only if $\delta_k$ is bounded from below.
\end{proposition}
\begin{proof} Since $d_{k} \circ d_{k-1} \vert_{L^{(k-1,p)}(X,E)} = 0$, $\mathrm{im} \, d_{k-1} \subseteq \ker d_k \vert_{L^{(k,p)}(X,E)}$ without further assumptions. Now, assume $\delta_k$ is bounded from below. Then $\delta_k$ is injective; towards a contradiction, suppose $f,g \in (\ker \, d_{k} \vert_{L^{(k,p)}(X,E)})^*$ such that $f \neq g$ and $\delta_k \, f = \delta_k \, g$. Recalling that $\Vert \cdot \Vert_{(k,p^*)}$ is a norm restricted to $L^{(k,p^*)}(X,E)$ leads to a contradiction
\begin{align}
0 = \Vert \delta_k f - \delta_k g \Vert_{(k-1,p^*)} = \Vert \delta_k (f-g)\Vert_{(k-1,p^*)} \geq K \Vert f - g\Vert_{(k,p^*)} > 0. \nonumber
\end{align}
Thus, $\delta$ is injective. In particular, $\ker \, \delta_k = \lbrace 0 \rbrace$ and since $L^{(k,p)}(X,E)$ is reflexive and $\ker \, d_k\vert_{L^{(k,p)}(X,E)}$ is closed, the latter is also reflexive and \begin{align} \mathrm{im} \, d_{k-1} &=  \mathrm{Ann}(\ker \, \delta_k) \nonumber \\ &= \lbrace f \in {(\ker \, d_{k} \vert_{L^{(k,p)}(X,E)})^*}^* \colon \langle g,f \rangle_{k} = 0, \, \forall g \in \ker \, \delta_k \rbrace \nonumber \\ &\cong \lbrace f \in \ker \, d_{k} \vert_{L^{(k,p)}(X,E)} \colon \langle g,f \rangle_{k} = 0, \, \forall g \in \lbrace 0 \rbrace \rbrace \nonumber \\ &= \ker \, d_{k} \vert_{L^{(k,p)}(X,E)}, \nonumber \end{align} so $d_{k-1}$ is onto $\ker d_k \vert_{L^{(k,p)}(X,E)}$. Next, suppose $d_{k-1}$ is onto. Since $d_{k-1}$ is bounded, it is bounded from below by the open mapping theorem if $d_{k-1}$ is injective. \end{proof}
This criteria is in fact related to the Poincar\'e constants of the links as Proposition \ref{bedlewoinequalityNOT} shows. This allows us to formulate a spectral condition for cohomological vanishing.
\begin{proposition} \label{FIX1}
Suppose $X$ is a $2$-dimensional locally finite simplicial complex such that for any vertex $\tau$ of $X$ the link $X_\tau$ is connected and $\mathcal{H}$ a separable infinite-dimensional Hilbert space.  Suppose there exists a constant $C$ such that   
\begin{enumerate}
\item[i.] the map $C \mathcal{I}: E \rightarrow \mathcal{H}$ where $\mathcal{I}: E \rightarrow \mathcal{H}$ is the identity map, is a Banach space isomorphism with the property $\Vert x \Vert_E \leq \Vert C\mathcal{I}(x) \Vert_\mathcal{H} \leq C \Vert x \Vert_E$ for all $x \in E$. Then, for $f \in L^{(1,2)}(X,E) \cap \ker {d}$ $$\kappa_2(X_\tau, \mathcal{H})^{-2} \Vert Mf_\tau\Vert^2_{(0,2)} + Q_\tau(f_\tau) \geq \dfrac{1}{C^2}\left( \kappa_2(X_\tau, \mathcal{H})^{-2} -  \dfrac{C^2}{2}\right) \Vert f_\tau \Vert^2_{(0,2)};$$ 
\item[ii.] the identity map $\overline{\mathcal{I}}: E^* \rightarrow \mathcal{H}$ is a Banach space isomorphism with the property $\Vert x \Vert_{E^*} \leq \Vert \overline{\mathcal{I}}(x) \Vert_{\mathcal{H}} \leq C \Vert x \Vert_{E^*}$ for all $x \in {E^*}$. Then, for $f \in L^{(1,2)}(X,E^*) \cap \ker \overline{d}$ $$\kappa_2(X_\tau, \mathcal{H})^{-2} \Vert \overline{M}f_\tau\Vert^2_{(0,2)} + Q_\tau(f_\tau) \geq \dfrac{1}{C^2}\left( \kappa_2(X_\tau, \mathcal{H})^{-2} -  \dfrac{C^2}{2}\right) \Vert f_\tau \Vert^2_{(0,2)}.$$
\end{enumerate}
\end{proposition}
\begin{proof}
$(i)$. As usual, write $\Vert \cdot \Vert_{(k,2)}$ for the seminorm on $\mathcal{E}^{(k,2)}(X,E)$ and write $\Vert \cdot \Vert_{(k,2), \mathcal{H}}$ for the seminorm on $\mathcal{E}^{(k,2)}(X,\mathcal{H})$. For $f \in L^{(1,2)}(X,E) \cap \ker {d}$
\begin{align}
Q_\tau(f_\tau) = \Vert d_\tau f_\tau \Vert_{(1,2)}^2 - \dfrac{1}{2} \Vert f_\tau \Vert^2_{(0,2)} \geq \Vert d_\tau f_\tau \Vert_{(1,2), \mathcal{H}}^2 - \dfrac{C^2}{2} \Vert f_\tau \Vert^2_{(0,2), \mathcal{H}} \nonumber
\end{align} 
On the other hand, by the Poincar\'e inequality and the Pythagorean identity $$\Vert d_\tau f_\tau \Vert_{(1,2), \mathcal{H}}^2 \geq \kappa_2(X_\tau, \mathcal{H})^{-2}\Vert f_\tau\Vert^2_{(0,2), \mathcal{H}} - \kappa_2(X_\tau, \mathcal{H})^{-2}\Vert Mf_\tau\Vert^2_{(0,2), \mathcal{H}}$$ Hence,
\begin{align}
Q_\tau(f_\tau) \geq \kappa_2(X_\tau, \mathcal{H})^{-2}\Vert f_\tau\Vert^2_{(0,2), \mathcal{H}} - \kappa_2(X_\tau, \mathcal{H})^{-2}\Vert Mf_\tau\Vert^2_{(0,2), \mathcal{H}} - \dfrac{C^2}{2} \Vert f_\tau \Vert^2_{(0,2), \mathcal{H}}, \nonumber 
\end{align} and in terms of the $\Vert \cdot \Vert_{(k,p)}$ norm 
\begin{align}
\kappa_2(X_\tau, \mathcal{H})^{-2}\Vert Mf_\tau\Vert^2_{(0,2)} + Q_\tau(f_\tau) \geq \dfrac{1}{C^2}\left( \kappa_2(X_\tau, \mathcal{H})^{-2} - \dfrac{C^2}{2} \right) \Vert f_\tau \Vert^2_{(0,2)}. \nonumber
\end{align} Similarly for $(ii)$.

\end{proof}

\begin{corollary} \label{bedlewoinequalityNOT}
Assuming Proposition \ref{FIX1} holds such that $\kappa_2(X_\tau,\mathcal{H}) \leq \kappa_2(X,\mathcal{H})$ for every link $X_\tau$ of $X$. Then, for $$\kappa_2(X,\mathcal{H}) < \sqrt{2}C^{-1},$$  $\delta$ and $\overline{\delta}$  are bounded from below.
\end{corollary}
\begin{proof}
By Proposition \ref{FIX1}(i)
\begin{align}
\kappa_2(X_\tau, \mathcal{H})^{-2}\Vert Mf_\tau\Vert^2_{(0,2)} + Q_\tau(f_\tau) \geq \dfrac{1}{C^2}\left( \kappa_2(X_\tau, \mathcal{H})^{-2} - \dfrac{C^2}{2} \right) \Vert f_\tau \Vert^2_{(0,2)}. \nonumber
\end{align} Thus, summing over the representatives $\tau \in \Sigma(0, \Gamma)$ gives, applying Propositions \ref{summing}, \ref{averagenorms}, and \ref{propQ} to the three terms respectively, that
\begin{align}
\Vert \delta f\Vert^2_{(0,2)} \geq \left( \dfrac{2 \kappa_2(X,\mathcal{H})}{C}\right)^2 \left( \kappa_2(X_\tau, \mathcal{H})^{-2} - \dfrac{C^2}{2} \right) \Vert f \Vert_{(1,2)}^2. \nonumber
\end{align} So, $\delta$ is bounded from below for $\kappa_2(X,\mathcal{H}) < {\sqrt{2}}{C^{-1}}$. Similarly for $\overline{\delta}$.
\end{proof}

\begin{theorem} \label{theorem}
Let $X$ be a locally finite $2$-dimensional simplicial complex, $\Gamma$ a discrete properly discontinuous group of automorphisms of $X$ and $\pi : \Gamma \rightarrow \mathrm{B}(\mathcal{H})$ a uniformly bounded  representation of $\Gamma$ on a separable infinite-dimensional Hilbert space $\mathcal{H}$. Suppose the link $X_\tau$ of every vertex $\tau$ of $X$ is connected and the associated Poincar\'e constants satisfy $$C < \dfrac{\sqrt{2}}{\kappa_2(X_\tau, \mathcal{H})}$$
for $C = \sup_{g \in \Gamma} \Vert \pi_g \Vert$. Then, $L^2H^1(X,\pi) = 0$.
\end{theorem}

\begin{proof}
Let $E$ be the Banach space $(\mathcal{H}, \Vert \cdot \Vert_E)$ where $\Vert \cdot\Vert_E = \sup_{g \in \Gamma}\Vert \pi_g(\cdot)\Vert_\mathcal{H}$. Now, $\pi$ is an isometric representation on $E$ and have the dual diagrams: \\
\xymatrix{ 
{\ker \, \overline{d}\,} \ar@{^{(}->}[r]^{\overline{i}}  \ar@/^3pc/[d]^{\delta \circ i^* \circ \overline{i}}  & {L^{(1,2)}(X,E^*)} \ar@{->>}[d]^{i^*} \\
{L^{(0,2)}(X,E^*)} \ar[u]^{\overline{d}} & {\left( \ker \, d \right)^*} \ar[l]^{\delta} \\ {L^{(0,2)}(X,E)} \ar[r]^{d} & {\ker d} \ar@{^{(}->}[d]^{i} \ar@/^2pc/[l]^{\overline{d}^* \circ \overline{i}^* \circ i} \\ {\left( \ker \, \overline{d} \right)^*} \ar[u]^{\overline{d}^*} & {L^{(1,2)}(X,E)} \ar@{->>}[l]^{\overline{i}^*} 
}
\\ We claim that $L^2H^1(X,E) = 0$, that is $d_0$ is onto $\ker d_1$. By Proposition \ref{VCH1} it is enough to prove that $d^*_0 = \delta_1$ is bounded from below on $(\ker {d}_1)^*$. Since Proposition \ref{FIX1} holds for $C = \sup_{g \in \Gamma} \Vert \pi_g \Vert$, and $ \kappa_2(X_\tau, \mathcal{H}) < \sqrt{2} C^{-1}$, it follows by Corollary \ref{bedlewoinequalityNOT} that $\delta_1$ is bounded from below when restricted to $\ker \overline{d}$. Hence, 
$\delta_1$ is bounded from below on the image of $i^* \circ \overline{i}$. Thus, if $i^* \circ \overline{i}$ is onto $(\ker d_1)^*$, then $\delta_1$ is bounded from below on $(\ker d_1)^*$ and ${\delta_1 \,}^* = d_0$ is onto, by which the claim follows. By a similar argument ${\overline{d^*_0}}$  restricted to $\ker d_1$ is bounded from below, and thus $${\overline{d}_0\,}^* \circ {\overline{i}\,}^* \circ i \colon \ker d_1 \rightarrow L^{(0,p)}(X,E)$$ is bounded from below. In particular, ${\overline{i}\,}^* \circ i$ is bounded from below and hence $({\overline{i}\,}^* \circ i)^* = i^* \circ \overline{i}$ is onto $(\ker d_1)^*$.  \end{proof}

\end{document}